\documentclass[pdflatex,sn-mathphys-num]{sn-jnl}% Math and Physical Sciences Numbered Reference Style

\usepackage{color}
\usepackage{amsmath,latexsym,amssymb,amsfonts,mathtools}
\usepackage{algorithm,algpseudocode}
\usepackage{epsf,psfrag}
\usepackage{xspace}
\usepackage{bbm}
\usepackage{rotating}
\usepackage{multirow}
\usepackage{framed}
\usepackage{subfig}
\usepackage{setspace}
\usepackage{verbatim}
\usepackage{mathrsfs}  
\usepackage{url}
\usepackage{bm}

\usepackage{tikz,tikz-cd}
\usetikzlibrary{positioning,arrows,decorations.pathmorphing,backgrounds,fit,calc,matrix,shapes,shadows,trees}

\usepackage{enumitem}
\setlist[enumerate]{leftmargin=*}
\setlist[description]{leftmargin=*}

\theoremstyle{thmstyleone}%
\newtheorem{theorem}{Theorem}%  meant for continuous numbers
\newtheorem{proposition}[theorem]{Proposition}% 
\newtheorem{lemma}[theorem]{Lemma}% 
\newtheorem{corollary}[theorem]{Corollary}% 
\newtheorem{assumption}[theorem]{Assumption}% 
\newtheorem{condition}[theorem]{Condition}% 

\theoremstyle{thmstyletwo}%

\theoremstyle{thmstylethree}%

\newcommand{\real}{\mathbb{R}}

\newcommand{\cK}{\mathcal{K}}
\newcommand{\cL}{\mathcal{L}}

\newcommand{\cR}{\mathcal{R}}

\newcommand{\risk}{\cR}
\newcommand{\bbe}{\mathbb{E}}

\newcommand{\riskenv}{\mathfrak{A}}
\newcommand{\CVaR}{\operatorname{AVaR}}

\newcommand{\proj}{\textup{proj}}

\newcommand{\dom}[1]{\textup{dom}\,#1}

\newcommand{\prox}{\textup{prox}}

\newcommand{\ared}{\textup{ared}_k}
\newcommand{\pred}{\textup{pred}_k}
\newcommand{\cred}{\textup{cred}_k}

\begin{document}

\title[Inexact Nonsmooth Trust Regions]{An Inexact Trust-Region Method for Structured Nonsmooth Optimization with Application to Risk-Averse Stochastic Programming}
\author[1]{\fnm{Drew P.} \sur{Kouri}}\email{dpkouri@sandia.gov}
\affil[1]{\orgdiv{Optimization and Uncertainty Quantification}, \orgname{Sandia National Laboratories}, \orgaddress{\street{P.O.\ Box 5800}, \city{Albuquerque}, \postcode{87185-1320}, \state{NM}, \country{USA}}}

\abstract{
  We develop a trust-region method for efficiently minimizing the sum of a
  smooth function, a nonsmooth convex function, and the composition of a
  convex Lipschitz continuous function with a smooth function. Optimization
  problems with this structure arise in numerous applications including
  risk-averse stochastic programming and subproblems for nonsmooth penalty
  nonlinear programming methods.  Our algorithm permits the use of inexact value
  and derivative information, enabling the solution of infinite-dimensional
  problems governed by, e.g., partial differential equations (PDEs). We prove
  global convergence of our method and under additional regularity assumptions,
  demonstrate that the sequence of iterates accumulates at a stationary point
  of our target problem.  We demonstrate our method's efficiency on two
  PDE-constrained optimization examples, showing that its performance is
  invariant to the PDE discretization size.
}

\keywords{Trust Regions, Nonsmooth Optimization, Saddle Point Problems, Risk-Averse Optimization, Stochastic Optimization, Adaptivity}
\pacs[MSC Classification]{49M29, 49M37, 65K10, 90C15, 93E20}

\maketitle

\section{Introduction}

We develop an efficient trust-region method for solving structured nonsmooth,
nonconvex optimization problems with the form
\begin{equation}\label{eq:optprob}
  \min_{x\in X} \; \{J(x)\coloneqq f_0(x) + \phi_1(f_1(x)) + \phi_0(x)\},
\end{equation}
where $X$ and $Y$ are real Hilbert spaces, $f_0:X\to\real$ and $f_1:X\to Y$ are
smooth functions, $\phi_0:X\to(-\infty,+\infty]$ is proper, closed and convex,
and $\phi_1:Y\to\real$ is convex and Lipschitz continuous. This class
of problems arises in various applications including sparse estimation and
learning \cite{tibshirani1996regression}, nonlinear programming with
$L^1$-penalties \cite{JNocedal_SJWright_2006a}, and simulating mechanical
systems subject to Tresca friction \cite{capatina2014variational}.

In addition to these applications, a critically important application of
our method is risk-averse optimization, which has the form
\begin{equation}\label{eq:riskmin}
  \min_{x} \; f(x) + \risk(F(x)) + \phi(x).
\end{equation}
In the context of \eqref{eq:riskmin}, $f$ is a smooth deterministic function,
$F$ is a smooth stochastic function (i.e., $F(x)$ is a random variable for each
$x$), $\phi$ is a nonsmooth convex function, and $\risk$ is a coherent risk
measure \cite{PArtzner_FDelbaen_JMEber_DHeath_1999a}. Risk-averse optimization
problems arise in numerous applications ranging from financial mathematics to
reliability engineering
\cite{dentcheva2024risk,AShapiro_DDentcheva_ARuszczynski_2014a}.
Although ubiquitous, these problems are often challenging to solve
because of the nonsmoothness arising in $\risk$ and $\phi$.  In fact,
traditional nonsmooth optimization methods are often intractable when
\eqref{eq:riskmin} is nonconvex and $F$ is computationally expensive to
evaluate, as is the case when $F$ involves the solution of a system of partial
differential equations (PDEs) with uncertain coefficients
\cite{alexanderian2017mean,antil2023ttrisk,beiser2023adaptive,chen2023performance,geiersbach2020stochastic,DPKouri_MHeinkenschloss_DRidzal_BGvanBloemenWaanders_2013a,DPKouri_MHeinkenschloss_DRidzal_BGvanBloemenWaanders_2014a,DPKouri_AShapiro_2018a,DPKouri_TMSurowiec_2016a,DPKouri_TMSurowiec_2018a,luo2025efficient}.
Recently, \cite{kouri2022primal} introduced the specialized primal-dual risk
minimization algorithm.  At each iteration of the primal-dual risk minimization
algorithm, one approximately solves a smoothed risk-averse optimization
problem, which can be expensive.  Moreover, it is unclear how to leverage
inexact evaluations of $f$, $F$ and its derivatives during the primal-dual
iteration.  Similar approaches like the interior point method tailored to
the average value-at-risk\footnote{Also called the conditional value-at-risk,
expected shortfall and superquantile.} introduced in \cite{garreis2021interior}
suffer from the same challenges.

To address these shortcomings, we introduce a new trust-region method that can
efficiently solve \eqref{eq:optprob} even when $f_0$, $f_1$ and their
derivatives are computationally expensive to evaluate. In this setting, it is
critical that {\em (i)} our algorithm exhibits rapid convergence and {\em (ii)}
can leverage inexact evaluations of $f_0$, $f_1$ and their derivatives while
maintaining strong convergence guarantees. For example, in PDE-constrained
optimization \cite{MHinze_RPinnau_MUlbrich_SUlbrich_2009}, evaluating $f_1$
often requires the discretization and iterative numerical solution of a system
of nonlinear PDEs---a cost that is compounded when the PDE has uncertain or
random coefficients. To address the algorithmic requirements {\em (i)} and
{\em (ii)} above, we extend the inexact proximal trust-region algorithm
developed in \cite{baraldi2022proximal}. The method in
\cite{baraldi2022proximal} is applicable for minimizing $f_0+\phi_0$ and is
provably convergent even when $f_0$ and its derivative are computed inexactly.
Moreover, the method exhibits local superlinear---even quadratic---convergence
rates under mild assumptions \cite{baraldi2024local}. Roughly speaking, our
algorithm applies the trust-region algorithm in \cite{baraldi2022proximal} to
\eqref{eq:optprob} using a modified trust-region subproblem model to account
for the composition $\phi_1\circ f_1$. By exploiting the structure of
\eqref{eq:optprob}, we are able to extend the convergence theory from
\cite{baraldi2022proximal} to guarantee convergence of our algorithm.

The remainder of the paper is structured as follows.  In
Section~\ref{sec:prelim}, we discuss the problem formulation, assumptions and
optimality conditions for \eqref{eq:optprob}.  In Section~\ref{sec:algo},
we introduce our inexact trust-region algorithm and describe efficient
ways to solve the trust-region subproblem in
Sections~\ref{sec:subprob}~and~\ref{sec:cp}. We prove convergence of our method
in Section~\ref{sec:conv} and demonstrate its numerical performance in
Section~\ref{sec:num}.

\section{Preliminaries}\label{sec:prelim}

Let $X$ and $Y$ be real Hilbert spaces.  We denote the inner product and
associated norm on $X$ by $(\cdot,\cdot)_X$ and $\|\cdot\|_X$, respectively,
and analogously for $Y$.  We denote the Banach space of bounded linear
operators mapping $X$ into $Y$ by $\cL(X,Y)$ and the topological dual space of
$X$ (i.e., the space of continuous linear functionals on $X$) by
$X^*\coloneqq\cL(X,\real)$.  To simplify the presentation, we identify the
dual space $X^*$ with $X$ via the Riesz representation theorem and we
denote $\cL(X)\coloneqq\cL(X,X)$.  With this convention, the adjoint of a
linear operator $A\in\cL(X,Y)$, denoted by $A^*\in\cL(Y,X)$, satisfies
$(Ax,y)_Y=(x,A^*y)_X$ for all $x\in X$ and $y\in Y$. Recall that an extended
real-valued function
$\psi:X\to[-\infty,+\infty]$ is proper if $\psi(x)>-\infty$ for all $x\in X$
and there exists at least one $x_0\in X$ for which $\psi(x_0) < +\infty$. When
$\psi$ is convex, we denote its proximity operator by $\prox_\psi:X\to X$, i.e.,
\[
  \prox_\psi(x) \coloneqq \operatorname*{arg\,min}_{x'\in X} \{\tfrac{1}{2}\|x'-x\|_X^2 + \psi(x')\},
\]
its subdifferential by $\partial\psi:X\rightrightarrows X$, i.e.,
\[
  \partial\psi(x) \coloneqq \{v\in X\,\vert\, \psi(x') \ge \psi(x) + (v,x'-x)_X\;\;\forall\,x'\in X\},
\]
and its Fenchel conjugate by $\psi^*:X\to(-\infty,+\infty]$, i.e.,
\[
  \psi^*(v) \coloneqq \sup_{x\in X}\{(v,x)_X - \psi(x)\}.
\]
If $\psi$ is proper, closed and convex, then $\psi^{**}=\psi$
\cite[Proposition~4.1]{IEkeland_RTemam_1999}.
We denote the effective domains of $\psi$ and $\partial\psi$ by
\[
  \dom{\psi}\coloneqq \{x\in X\,\vert\, \psi(x) < +\infty\}
  \qquad\text{and}\qquad
  \dom{\partial\psi}\coloneqq \{x\in X\,\vert\, \partial\psi(x) \neq \emptyset\},
\]
respectively. Finally, for a nonempty, closed and convex subset $C\subseteq X$,
we denote the indicator and support functions of $C$ by $\delta_C$ and
$\sigma_C$, respectively, i.e.,
\[
  \delta_C(x) \coloneqq \left\{\begin{array}{ll}
    0 & \text{if $x\in C$} \\
    +\infty & \text{otherwise}
  \end{array}\right.
  \qquad\text{and}\qquad
  \sigma_C(x) \coloneqq \sup_{v\in C}\; (v,x)_X.
\]
Recall that $\sigma_C=\delta_C^*=\sigma_C^{**}$, the proximity operator of
$\delta_C$ is the metric projection onto $C$, which we denote by $\proj_C$, and
that the subdifferential of $\delta_C$ is the normal cone
\[
  N_C(x)\coloneqq\left\{\begin{array}{ll}
    \{\,v\in X\,\vert\,(v,x'-x)_X \le 0\;\;\forall\,x'\in C\} & \text{if $x\in C$} \\
    \emptyset & \text{otherwise.}
  \end{array}\right.
\]

We make the following assumptions on the problem data in \eqref{eq:optprob}.
\begin{assumption}\label{as:data}
  The problem data in \eqref{eq:optprob} satisfies the following conditions.
  \begin{enumerate}
  \item The function $\phi_0:X\to(-\infty,+\infty]$ is proper, closed and convex.
  \item The function $f_0:X\to\real$ is $M_{f_0}$-smooth on $\dom{\phi_0}$,
        i.e., there exists an open set $U$ containing $\dom{\phi_0}$ on
        which $f_0$ is Fr\'{e}chet differentiable and $\nabla f_0$ is Lipschitz
        continuous with modulus $M_{f_0}>0$.
  \item The function $\phi_1:Y\to\real$ is convex and Lipschitz continuous with
        modulus $M_{\phi_1}>0$.
  \item The function $f_1:X\to Y$ is $M_{f_1}$-smooth on $\dom{\phi_0}$,
        i.e., there exists an open set $V$ containing $\dom{\phi_0}$ on
        which $f_1$ is Fr\'{e}chet differentiable and $f_1'$ is Lipschitz
        continuous with modulus $M_{f_1}>0$.
  \item The objective function $J=f_0+\phi_1\circ f_1+\phi_0$ is bounded
        from below by $\kappa_{\text{lb}}\in\real$.
  \end{enumerate}
\end{assumption}

Under Assumption~\ref{as:data}, $f_0$ and $f_1$ are strictly differentiable
and locally Lipschitz continuous \cite[Proposition~2.2.1]{FHClarke_1983}.
Moreover, since $\phi_1$ is convex and Lipschitz continuous, it is lower
regular on $Y$ \cite[Definition~1.91]{mordukhovich2018variational}.  Therefore,
\cite[Theorem~3.41(ii)]{mordukhovich2018variational} ensures that
\[
  \partial (f_0+\phi_1\circ f_1)(x)
    = \nabla f_0(x) + f_1'(x)^*\partial\phi_1(f_1(x))
\]
for $x\in U\cap V$ and $f_0+\phi_1\circ f_1$ is lower regular. Here,
$\partial(f_0+\phi_1\circ f_1)$ denotes the basic subdifferential of
$f_0+\phi_1\circ f_1$ \cite[Definition~1.77]{mordukhovich2018variational} since
$f_0+\phi_1\circ f_1$ need not be convex. It further follows from
\cite[Proposition~1.92(ii)~\&~Theorem~1.93]{mordukhovich2018variational} that
$\phi_0$ is lower regular at any point $x\in\dom\phi_0$. Hence, the prior
derivation and \cite[Theorem~3.36]{mordukhovich2018variational} ensure that
\begin{equation}\label{eq:Jsubdiff}
  \partial J(x)
    = \nabla f_0(x) + f_1'(x)^*\partial\phi_1(f_1(x)) + \partial\phi_0(x)
    \neq \emptyset
\end{equation}
for $x\in\dom{\partial\phi_0}$.  Hence,
\cite[Proposition~1.114]{mordukhovich2018variational}
provides a first-order necessary optimality condition for \eqref{eq:optprob}.
In particular, if $\bar{x}\in X$ is a local minimizer of \eqref{eq:optprob},
then $\bar{x}$ satisfies
\begin{equation}\label{eq:statpt0}
  \exists\,\bar{\theta}\in\partial\phi_1(f_1(\bar{x}))
  \quad\text{such that}\quad
  -(\nabla f_0(\bar{x}) + f_1'(\bar{x})^*\bar{\theta}) \in \partial\phi_0(\bar{x}),
\end{equation}
which can be equivalently rewritten as
\begin{equation}\label{eq:statpt}
  \exists\,\bar{\theta}\in\partial\phi_1(f_1(\bar{x}))
  \quad\text{such that}\quad
  \bar{x} = \prox_{t\phi_0}(\bar{x}-t(\nabla f_0(\bar{x}) + f_1'(\bar{x})^*\bar{\theta}))
\end{equation}
for arbitrary fixed $t>0$.  We will say that $\bar{x}\in X$ is a stationary
point of \eqref{eq:optprob} if it satisfies \eqref{eq:statpt}.

For the forthcoming analysis, it will be convenient to define the proper, closed
and convex function $\psi_{\bar{x}}:X\to(-\infty,+\infty]$ for fixed
$\bar{x}\in X$ defined by
\begin{equation}\label{eq:psi}
  \psi_{\bar{x}}(x)\coloneqq \phi_1(f_1'(\bar{x})(x-\bar{x})+f_1(\bar{x}))+\phi_0(x).
\end{equation}
Using $\psi_{\bar{x}}$, we arrive at the following alternative characterization
of stationary points for \eqref{eq:optprob}.
\begin{theorem}\label{T:statpt}
  Let $\bar x\in X$ and consider $\psi_{\bar x}$ defined in \eqref{eq:psi}.
  Then, $\bar x$ is a stationary point of \eqref{eq:optprob} (i.e.,
  \eqref{eq:statpt0} holds) if and only if
  \begin{equation}\label{eq:linprox}
    \bar{x}=\prox_{t\psi_{\bar{x}}}(\bar{x}-t\nabla f_0(\bar{x}))
  \end{equation}
  for arbitrary, fixed $t>0$.  In particular, if $\bar{x}$ is a stationary
  point of \eqref{eq:optprob}, then it is also a stationary point of the
  auxiliary optimization problem
  \begin{equation}\label{eq:lin}
    \min_{x\in X} f_0(x) + \psi_{\bar{x}}(x).
  \end{equation}
\end{theorem}
\begin{proof}
  Leveraging the same arguments that produced \eqref{eq:Jsubdiff}, we obtain
  \[
    \partial\psi_{\bar x}(x) = f_1'(\bar x)^*\partial\phi_1(f_1'(\bar x)(x-\bar x)+f_1(\bar x)) + \partial\phi_0(x).
  \]
  From this follows the sequence of equivalences:
  \[
  \begin{aligned}
    \eqref{eq:statpt0} \qquad&\iff\qquad -\nabla f_0(\bar x) \in f_1'(\bar x)^*\partial\phi_1(f_1(\bar x))+\partial\phi_0(\bar x)=\partial\psi_{\bar x}(\bar x) \\
    \qquad&\iff\qquad \bar x - t\nabla f_0(\bar x) \in (\text{Id}+t\partial\psi_{\bar x})(\bar x) \quad\forall\, t>0 \\
    \qquad&\iff\qquad \eqref{eq:linprox}, \qquad\textup{cf.\ \cite[Example~23.3]{bauschke2017convex}}.
  \end{aligned}
  \]
  Here, $\text{Id}\in\cL(X)$ denotes the identity operator on $X$. To conclude,
  a stationary point $\bar{p}\in X$ for \eqref{eq:lin} satisfies
  \begin{equation}\label{eq:lin-optcond}
    -\nabla f_0(\bar{p}) \in \partial\psi_{\bar{x}}(\bar{p})
    = f_1'(\bar{x})^*\partial\phi_1(f_1'(\bar{x})(\bar{p}-\bar{x})+f_1(\bar{x})) + \partial\phi_0(\bar{p}).
  \end{equation}
  Since $\bar{x}$ is a stationary point of \eqref{eq:optprob}, it
  satisfies \eqref{eq:statpt0} and therefore, substituting $\bar{x}$ for
  $\bar{p}$ in \eqref{eq:lin-optcond}, we see that $\bar{x}$ is a stationary
  point of \eqref{eq:lin}.
\end{proof}

Our trust-region algorithm leverages nonsmooth functions of the form
\[
  \psi_k(x) = \phi_1(A_k(x-x_k)+b_k) + \phi_0(x),
\]
where $x_k$ is the $k$-th iterate, $A_k\in\mathcal{L}(X,Y)$
is an approximation of $f_1'(x_k)$ and $b_k\in Y$ is an approximation of
$f_1(x_k)$ To prove convergence, we must relate the proximity operators of
$\psi_k$ and $\psi_{\bar{x}}$, which we do in the subsequent technical lemma.

\begin{lemma}\label{L:proxDiff}
  Let $w_i:X\to Y$, $i=1,2$, be two affine maps defined by
  \[
    w_i(x) = D_i(x-u_i) + d_i, \quad i=1,2,
  \]
  where $D_i\in\mathcal{L}(X,Y)$, $u_i\in X$ and $d_i\in Y$, $i=1,2$.
  Moreover, define $\Psi_i:X\to(-\infty,+\infty]$ by
  \[
    \Psi_i(x) = \phi_0(x) + \phi_1(w_i(x)), \quad i=1,2.
  \]
  If $p_i=\prox_{t\Psi_i}(z)$, $i=1,2$, for fixed $t>0$ and $z\in X$, then
  \begin{equation}\label{eq:proxDiff-0}
  \begin{aligned}
    &\frac{\|p_1-p_2\|_X^2}{\max\{1,\|p_1-p_2\|_X\}} \\
    &\le 2 t M_{\phi_1}(\|D_1-D_2\|_{\cL(X,Y)}(1+2\|p_1-u_1\|_X)+2\|d_1-d_2-D_2(u_1-u_2)\|_Y).
  \end{aligned}
  \end{equation}
\end{lemma}
\begin{proof}
  Let $v_i:X\to(-\infty,+\infty]$ denote the objective function associated with
  the proximity operator of $\Psi_i$, i.e.,
  \[
    v_i(x) = \tfrac{1}{2t}\|x-z\|_X^2 + \Psi_i(x), \quad i=1,2,
  \]
  and notice that, for $x\in\dom{\phi_0}$,
  \[
    v_1(x)-v_2(x) = \phi_1(w_1(x))-\phi_1(w_2(x)).
  \]
  Owing to the strong convexity of $v_2$ as well as the optimality of
  $p_i$, $i=1,2$, we have that
  \begin{align}
    \tfrac{1}{2t}\|p_2-p_1\|_X^2 &\le v_2(p_1)-v_2(p_2) \nonumber \\
      &\le [(v_1(p_2)-v_2(p_2))-(v_1(p_1)-v_2(p_1))] + (v_1(p_1)-v_1(p_2)) \nonumber \\
      &\le [(v_1(p_2)-v_2(p_2))-(v_1(p_1)-v_2(p_1))].
      \label{eq:proxDiff-1}
  \end{align}
  It follows from the definition of the Fenchel conjugate of $\phi_1$ and the
  Fenchel-Young inequality that
  \[
    (\theta_2(x),w_1(x)-w_2(x))_Y \le v_1(x)-v_2(x) \le (\theta_1(x),w_1(x)-w_2(x))_Y
  \]
  for any $x\in X$ and $\theta_i(x)\in\partial\phi_1(w_i(x))$,
  $i=1,2$. Consequently, we have that
  \begin{subequations}
  \begin{align}
    [(v_1(x)-v_2(x))-&(v_1(y)-v_2(y))] \label{eq:proxDiff-2a} \\
    \le&\, (\theta_1(x),w_1(x)-w_2(x))_Y - (\theta_2(y),w_1(y)-w_2(y))_Y \nonumber \\
      =&\, (\theta_1(x),(w_1(x)-w_2(x))-(w_1(y)-w_2(y)))_Y \label{eq:proxDiff-2b} \\
       &+ (\theta_1(x)-\theta_2(y),w_1(y)-w_2(y))_Y \label{eq:proxDiff-2c}
  \end{align}
  \end{subequations}
  for any $x,\,y\in\dom{\phi_0}$. To bound \eqref{eq:proxDiff-2b}, the
  Lipschitz continuity of $\phi_1$ and the definitions of $w_i$ yield
  \[
    |(\theta_1(x),(w_1(x)-w_2(x))-(w_1(y)-w_2(y)))_Y| \le M_{\phi_1}\|D_1-D_2\|_{\cL(X,Y)}\|x-y\|_X.
  \]
  Here, recall that $\|\theta\|_Y\le M_{\phi_1}$ for all
	$\theta\in\partial\phi_1(y)$ for $y\in Y$ \cite[Proposition~2.1.2(a)]{FHClarke_1998a}.  Now, to bound
  \eqref{eq:proxDiff-2c}, we have that
  \[
    w_1(y)-w_2(y) = (D_1-D_2)(y-u_1)+(d_1-d_2-D_2(u_1-u_2)),
  \]
  and so
  \[
    \begin{aligned}
    |(\theta_1(x)&-\theta_2(y),w_1(y)-w_2(y))_Y| \\
    &\le 2M_{\phi_1}(\|D_1-D_2\|_{\cL(X,Y)}\|y-u_1\|_X + \|d_1-d_2-D_2(u_1-u_2)\|_Y).
    \end{aligned}
  \]
  Using these estimates, we can bound \eqref{eq:proxDiff-2a} by
  \begin{align}
    |&(v_1(x)-v_2(x))-(v_1(y)-v_2(y))| \nonumber \\
     &\le M_{\phi_1}(\|D_1-D_2\|_{\cL(X,Y)}(\|x-y\|_X+2\|y-u_1\|_X)+2\|d_1-d_2-D_2(u_1-u_2)\|_Y).
     \label{eq:proxDiff-3}
  \end{align}
  Combining \eqref{eq:proxDiff-1} with \eqref{eq:proxDiff-3}, we achieve the
  bound
  \[
    \tfrac{1}{2t}\|p_2-p_1\|_X^2
    \le M_{\phi_1}(\|D_1-D_2\|_{\cL(X,Y)}(\|p_2-p_1\|_X+2\|p_1-u_1\|_X)+2\|d_1-d_2-D_2(u_1-u_2)\|_Y).
  \]
  Applying H\"older's inequality to the right-hand side yields the bound
  \eqref{eq:proxDiff-0}.
\end{proof}

\section{Algorithm}\label{sec:algo}
We now present an iterative method for solving \eqref{eq:optprob}.  At each
iteration of our method, we minimize a local model of the objective function
$J$ in \eqref{eq:optprob} within a ball of radius $\Delta_k>0$ called the
{\em trust region}.  To construct this subproblem, we employ a quadratic
approximation of the Lagrangian functional
\[
  L(x,\theta) \coloneqq f_0(x) + (\theta, f_1(x))_Y.
\]
Notice that $L$ is continuously differentiable and that
$x\mapsto L(x,\theta)$ has Lipschitz continuous gradient.  We approximate
$L$ around the current iterate $x_k$ by
\[
  L_k(x,\theta) \coloneqq q_k(x) + (\theta,\ell_k(x))_Y
\]
and
\[
  \begin{aligned}
    q_k(x) &\coloneqq \frac{1}{2} (B_k(x-x_k), x-x_k)_X + (g_k,x-x_k)_X \\
    \ell_k(x) &\coloneqq A_k(x-x_k) + b_k.
  \end{aligned}
\]
Here, $A_k\approx f_1'(x_k)$, $b_k\approx f_1(x_k)$,
$g_k\approx\nabla f_0(x_k)$, and $B_k\in\mathcal{L}(X)$ encapsulates the
curvature of $L(\cdot,\theta_k)$ at $x_k$ for some $\theta_k\in Y$.  Using
$L_k$, we approximate the objective function $J$ by the local model
\begin{equation}\label{eq:quadmod_L}
  m_k(x)\coloneqq \sup_{\theta\in Y} \{L_k(x,\theta) - \phi_1^*(\theta)\} + \phi_0(x)
                = q_k(x) + \psi_k(x),
\end{equation}
where
\[
  \psi_k(x) \coloneqq \phi_1(\ell_k(x))+\phi_0(x).
\]
Note that the function $\psi_k$ involves the sum of two potentially nonsmooth,
yet convex, terms $\phi_1\circ\ell_k$ and $\phi_0$.  To compute trial
iterates $x_k^+$, we approximately solve the trust-region subproblems
\begin{equation}\label{eq:tr-sub}
  \min_{x\in X} \; m_k(x)
  \quad\text{subject to}\quad
  \|x-x_k\|_X \le \Delta_k,
\end{equation}
where $\Delta_k>0$ is the current trust-region radius.

To ensure convergence, we require that for each $k$ the trial iterate $x_k^+$
satisfies
\begin{subequations}\label{eq:trial}
\begin{align}
  \|x_k^+-x_k\|_X &\le \kappa_{\rm rad}\Delta_k \label{eq:trcon} \\
  m_k(x_k) - m_k(x_k^+) &\ge \kappa_{\rm fcd} h_k \min\left\{\frac{h_k}{1+\|B_k\|_{\mathcal{L}(X)}},\Delta_k\right\}, \label{eq:fcd}
\end{align}
\end{subequations}
where $\kappa_{\rm rad}>0$ and $\kappa_{\rm fcd}>0$ are independent of $k$ and
$h_k$ is the stationarity metric
\begin{equation}\label{eq:hk}
  h_k \coloneqq \frac{1}{t_k}\|\prox_{t_k\psi_k}(x_k-t_k g_k) - x_k\|_X, \quad t_k > 0.
\end{equation}
Note that \eqref{eq:fcd} ensures that $x_k^+\in\dom{\phi_0}$ since otherwise the
left-hand side would be $-\infty$.  The positive parameter $t_k > 0$ in
\eqref{eq:hk} is chosen as the step size for the Cauchy point discussed in
Section~\ref{sec:cp}. Once $x_k^+$ is computed, we decide whether or not
to accept or reject $x_k^+$ based on the ratio of actual and predicted
reduction:
\[
  \rho_k^* \coloneqq \frac{\ared}{\pred},
\]
where
\[
  \ared \coloneqq J(x_k)-J(x_k^+) \qquad\text{and}\qquad
  \pred \coloneqq m_k(x_k)-m_k(x_k^+).
\]
In particular, if $\rho_k^*\ge\eta_1$, then $x_{k+1}=x_k^+$ and otherwise
$x_{k+1}=x_k$.  Moreover, we decrease the trust-region radius when
$\rho_k^*<\eta_1$ and increase it when $\rho_k^*>\eta_2$.  Here,
$0<\eta_1<\eta_2<1$ are user-specified parameters, which we set to
$\eta_1=10^{-4}$ and $\eta_2=0.5$ in our numerical results.

In many applications, the values and derivatives of $f_0$ and $f_1$ cannot be
numerically computed exactly.  For optimization problems governed by systems
of PDEs, $f_0$ and $f_1$ may require the discretization and iterative solution
of the underlying PDEs \cite{MHinze_RPinnau_MUlbrich_SUlbrich_2009}.
Similarly, in risk-averse stochastic programming, evaluating the risk measure
requires approximation, which can be done in some cases using quadrature
\cite{chen2015scenario,heinkenschloss2025optimization,DPKouri_MHeinkenschloss_DRidzal_BGvanBloemenWaanders_2013a,DPKouri_MHeinkenschloss_DRidzal_BGvanBloemenWaanders_2014a}.
To leverage this inexactness, we enforce the following conditions on $g_k$,
$b_k$ and $A_k$.
\begin{condition}\label{as:inexactGrad}
  There exist positive constants $\kappa_{\rm grad}$, $\kappa_{\rm val}$
  and $\kappa_{\rm jac}$, independent of $k$, such that
  \begin{subequations}
  \begin{align}
    \|g_k - \nabla f_0(x_k)\|_X &\le \kappa_{\rm grad}\min\{h_k,\Delta_k\} \label{eq:inexactGrad1} \\
    \|b_k - f_1(x_k)\|_Y &\le \kappa_{\rm val}\min\{h_k,\Delta_k^2\} \label{eq:inexactGrad2} \\
    \|A_k - f_1'(x_k)\|_{\cL(X,Y)} &\le \kappa_{\rm jac}\min\{h_k,\Delta_k\}. \label{eq:inexactGrad3}
  \end{align}
  \end{subequations}
\end{condition}
In these settings, the objective function value $J$, and thus $\ared$, cannot
be computed exactly. Instead, we replace $\ared$ with the {\em computed reduction}
$\cred$.  We enforce the following condition on the approximation $\cred$.
\begin{condition}\label{as:inexactObj}
  There exists a positive constant $\kappa_{\rm obj}$, independent of $k$, such that
  \begin{equation}
    |\ared-\cred|\le\kappa_{\rm obj}[\eta\min\{\pred,\zeta_k\}]^\zeta,
  \end{equation}
  where $\zeta$, $\eta$, and $\zeta_k$ are (user-specified) positive real numbers
  that satisfy
  \[
    \zeta > 1, \qquad 0 < \eta < \min\{\eta_1,1-\eta_2\},\qquad\text{and}\qquad
    \lim_{k\to\infty}\zeta_k = 0.
  \]
\end{condition}
Using $\cred$, we replace the ratio of actual and predicted reduction
$\rho_k^*$ with
\[
  \rho_k \coloneqq \frac{\cred}{\pred}.
\]
As shown in \cite[Lemma~A.1]{DPKouri_MHeinkenschloss_DRidzal_BGvanBloemenWaanders_2014a},
Condition~\ref{as:inexactObj} ensures that
\begin{equation}\label{eq:K-eta}
  \exists\,K_\eta\in\mathbb{N}\qquad\text{such that}\qquad |\rho_k^*-\rho_k| \le \eta \qquad\forall\,k\ge K_\eta.
\end{equation}
We list the trust-region method in Algorithm~\ref{alg:tr} and in the
forthcoming subsections, we discuss methods to efficiently compute the
trial iterate $x_k^+$ in line~3 of Algorithm~\ref{alg:tr}.

\begin{algorithm}[!ht]
\caption{Nonsmooth Trust-Region Algorithm}
  \label{alg:tr}
  \begin{algorithmic}[1]
    \Require{Initial guess $x_1\in\dom{\phi_0}$, initial radius
    $\Delta_1>0$, $0<\eta_1<\eta_2<1$, and $0<\gamma_1\le\gamma_2<1\le \gamma_3$}
    \For{$k=1,2,\ldots$}
      \State{{\bf Model Selection:} Choose $A_k$, $b_k$, and $g_k$ that satisfy
            Condition~\ref{as:inexactGrad}}
      \State{{\bf Step Computation:} Compute $x_k^+\in X$ that satisfies
            \eqref{eq:trial}}
      \State{{\bf Step Acceptance and Radius Update:} Compute the ratio of
	    computed and predicted reduction $\rho_k$ with $\cred$ satisfying
            Condition~\ref{as:inexactObj}}
      \If{$\rho_k < \eta_1$}
        \State{Set $x_{k+1} \gets x_k$}
        \State{Choose $\Delta_{k+1}\in[\gamma_1\Delta_k,\gamma_2\Delta_k]$}
      \Else
        \State{Set $x_{k+1}\gets x_k^+$}
        \If{$\rho_k\in[\eta_1,\eta_2)$}
          \State{Choose $\Delta_{k+1}\in[\gamma_2\Delta_k,\Delta_k]$}
	\Else
	  \State{Choose $\Delta_{k+1}\in[\Delta_k,\gamma_3\Delta_k]$}
	\EndIf
      \EndIf
    \EndFor
  \end{algorithmic}
\end{algorithm}

\subsection{Trial Iterate Computation}\label{sec:subprob}
To compute a trial step $x_k^+$, we first compute a benchmark $x_k^c$ called
the Cauchy point, which we describe in detail in the next subsection.
Importantly, the Cauchy
point satisfies \eqref{eq:trial}.  We then improve upon the Cauchy point ---
maintaining the satisfaction of \eqref{eq:trial} --- using a descent method
such as the methods introduced in \cite{baraldi2024efficient}, specifically the
dogleg, spectral proximal gradient (SPG), or truncated conjugate gradient (CG)
algorithms. For our numerical results, we employ the truncated CG method
\cite[Algorithm~4]{baraldi2024efficient}. This method first computes the Cauchy
point $x_k^c$ and then improves upon it using modified nonlinear CG iterations.
Each iteration of truncated CG, as well as the SPG methods in
\cite{baraldi2022proximal,baraldi2024efficient}, requires a single application
of the model Hessian $B_k$ to a vector.  Additionally, these methods are
terminated if either {\em (i)} the norm of the proximal gradient falls below
a relative stopping tolerance, {\em (ii)} the maximum number of iterations is
reached, or {\em (iii)} the trust-region radius is exceeded.

The efficiency, and likely the local convergence rate, of Algorithm~\ref{alg:tr}
depends on the choice of the model Hessian $B_k$, which can be selected in a
variety of ways.  Basing the choice of $B_k$ on the Hessian of
$L(x_k,\theta_k)$ permits the inclusion of curvature information from $f_1(x_k)$,
which would otherwise be neglected. For example, when
both $f_0$ and $f_1$ are twice differentiable, we can choose
$B_k=\nabla^2 L(x_k,\theta_k)$, assuming the applications of the Hessians of
$f_0$ and $f_1$ are computable.  Here, the choice
$\theta_k\in\partial\phi_1(f_1(x_k))$ can provide additional sensitivity information
from $\phi_1$ even when it is nonsmooth. For many inexact and adaptive
approximations of $f_0$ and $f_1$, we often not only approximate their values and
derivatives at $x_k$, but rather produce surrogate models of $f_0$ and $f_1$,
which we can use to compute $B_k$.  This is the case in PDE-constrained
optimization, when $f_0$ and $f_1$ are approximated using adaptive discretizations
or reduced-order models, and is the choice we make in our numerical experiments.
When Hessians are unavailable, we can approximate them using secant methods as
is commonly done in nonlinear programming software
\cite[Section~2.9]{gill2005snopt}.  Suppose $x_{k+1}=x_k^+$ and let
$s_k=x_{k+1}-x_k$, $u_k=\nabla L(x_{k+1},\theta_{k+1})-\nabla L(x_k,\theta_{k+1})$
and $p_k=B_ks_k$, then $B_k$ can be updated using, e.g., the BFGS update,
which is defined by its application to a vector $v\in X$:
\[
  B_{k+1}v = B_kv + \frac{(u_k,v)_X}{(u_k,s_k)_X}u_k - \frac{(p_k,v)_X}{(p_k,s_k)_X}p_k.
\]
Here, $\theta_{k+1}\in Y$ is an approximation of the dual variables
corresponding to $\phi_1(f_1(x_{k+1}))$.  Reasonable choices for $\theta_{k+1}$
include $\theta_{k+1}\in\partial\phi_1(f_1(x_{k+1}))$ and the dual variables
associated with the Cauchy point, which are computed in
Algorithm~\ref{alg:prox-phik}.  We describe this procedure next.

\subsection{Cauchy Point}\label{sec:cp}

In traditional trust-region methods, the trial iterate $x_k^+$ is selected
to produce at least a fraction of the model decrease achieved by
the Cauchy point.  For smooth unconstrained optimization, the
Cauchy point is chosen as a point in the negative gradient direction that
produces {\em sufficient} decrease of the (usually quadratic) model
\cite{ARConn_NIMGould_PhLToint_2000a}, while for convex-constrained
trust-region methods, the Cauchy point is a point along the projected gradient
path \cite{lin1999newton}. For this work, we define the Cauchy point for
\eqref{eq:tr-sub} as
\begin{equation}\label{eq:cp}
  x_k^c \coloneqq x_k + \alpha_k s_k(t_k)
  \qquad\text{with}\qquad
  s_k(t) \coloneqq \prox_{t\psi_k}(x_k - tg_k)-x_k,
\end{equation}
where $\alpha_k\in(0,1]$ and $t_k>0$ are selected to satisfy \eqref{eq:trial}.
When minimizing $f_0+\phi_0$, two basic approaches for selecting $\alpha_k$ and
$t_k$ are presented in \cite{baraldi2022proximal,baraldi2024efficient}.  These
approaches apply, without modification, to \eqref{eq:tr-sub}.  The approach
described in \cite[Algorithm~2]{baraldi2022proximal} sets $\alpha_k=1$ and
selects the step length $t_k$ to satisfy certain sufficient decrease conditions
using a bidirectional proximal search.  In contrast, the method described in
\cite{baraldi2024efficient} first computes $t_k$ as a safeguarded spectral
(also called Barzilai-Borwein) step length and then selects $\alpha_k$
to minimize a quadratic upper bound of the model. In our numerical experiments,
we employ the approach described in \cite{baraldi2024efficient}.

In order to compute $x_k^c$, we must evaluate the proximity operator
of $\psi_k$.  However, $\psi_k$ involves the sum of two convex functions, one
of which is composed with an affine map.  As a result, $\psi_k$ is
unlikely to admit an analytical proximity operator.  However, we are able to
exploit the structure of $\psi_k$ if the proximity operators of $\phi_0$ and
$\phi_1^*$ are available.  In this case, we can approximately compute the
proximity operator of $\psi_k$ iteratively by applying, e.g., projected
gradient ascent to the dual problem. To this end, consider the optimization problem for
evaluating $\prox_{r\psi_k}(z)$ with $z\in X$ and $r>0$, i.e.,
\begin{equation}\label{eq:primalprox}
  \min_{x\in X} \{\tfrac{1}{2r} \|x-z\|_X^2 + \phi_0(x) + \phi_1(A_k(x-x_k)+b_k)\}.
\end{equation}
The dual problem associated with \eqref{eq:primalprox} is given by
\begin{align}
  \max_{\theta\in Y} \min_{x\in X} \{\tfrac{1}{2r}\|x-z\|_X^2 + \phi_0(x) + (\theta,A_k(x-x_k)+b_k)_Y - \phi_1^*(\theta)\}.\label{eq:dualprox}
\end{align}
Since $\phi_0$ is proper and $\phi_1$ is continuous,
\cite[Proposition~15.24(vii)~\&~Theorem~15.23]{bauschke2017convex} ensure that
there is no duality gap between \eqref{eq:primalprox} and \eqref{eq:dualprox}.
Moreover, the inner minimization problem in \eqref{eq:dualprox} has the unique
minimizer $\bar{p}(\theta)$, which satisfies the optimality conditions
\[
  \begin{aligned}
  0 \in \tfrac{1}{r}(\bar{p}(\theta)-z) + A_k^*\theta + \partial\phi_0(\bar{p}(\theta))
  &\quad\iff\quad
  (z-rA_k^*\theta) \in \bar{p}(\theta)+r\partial\phi_0(\bar{p}(\theta)) \\
  &\quad\iff\quad
  \bar{p}(\theta) = \prox_{r\phi_0}(z-rA_k^*\theta).
  \end{aligned}
\]
Substituting $\bar{p}(\theta)$ into \eqref{eq:dualprox} yields the dual
objective function $v(\theta)\coloneqq d(\theta)-\phi_1^*(\theta)$, where
\[
  d(\theta) \coloneqq
    \tfrac{1}{2r}\|\bar{p}(\theta)-z\|_X^2 + \phi_0(\bar{p}(\theta)) + (\theta,A_k(\bar{p}(\theta)-x_k)+b_k)_Y.
\]
Since $d$ is the value function of an affinely perturbed minimization problem,
it is concave.  In addition, $d$ can be rewritten as
\[
  d(\theta) = \tfrac{1}{2r}\|\bar{p}(\theta)-(z-rA_k^*\theta)\|_X^2 + \phi_0(\bar{p}(\theta)) + (\theta,A_k(z-x_k)+b_k)_Y - \tfrac{r}{2}\|A_k^*\theta\|_Y^2,
\]
where the first two terms constitute the Moreau envelope of $\phi_0$
evaluated at $z-rA_k^*\theta$.  Consequently, $d$ is differentiable
with Lipschitz continuous gradient given by
\[
  \nabla d(\theta) = A_k(\bar{p}(\theta)-x_k)+b_k
  = \ell_k(\bar{p}(\theta))
\]
and Lipschitz modulus $r\|A_k\|_{\cL(X,Y)}^2$
\cite[Proposition~12.30]{bauschke2017convex}.  One can maximize the dual
objective function
$$V(\theta)\coloneqq d(\theta)-\phi_1^*(\theta)$$
using, e.g., the spectral proximal gradient method that we list in
Appendix~\ref{ap:spg} (see also \cite{birgin2000nonmonotone}).  Algorithm~\ref{alg:prox-phik} lists the full
algorithm for evaluating $\prox_{r\psi_k}(z)$ for $z\in X$.   We note that
each iteration requires a single application of the Jacobian $A_k$ and its
adjoint $A_k^*$.
\begin{algorithm}[!ht]
\caption{$\psi_k$ Proximity Operator Computation}
  \label{alg:prox-phik}
  \begin{algorithmic}[1]
    \Require{Proximity operator arguments $z\in X$ and $r>0$, initial guess
      $\theta^{(1)}\in\dom{\phi_1^*}$ and dual function $V^{(1)}=V(\theta^{(1)})$,
      safeguards $0<\gamma_{\min}<\gamma_{\max}<+\infty$ and $a_{\min}\in(0,1]$,
      initial step lengths $\gamma^{(1)}\in[\gamma_{\min},\gamma_{\max}]$ and
      $\lambda_0\in(0,1]$, and $0<\sigma_1<\sigma_2<1$, and positive parameters
      $\alpha\in(0,1)$ and $\beta\in[\sigma_1,\sigma_2]$}
    \For{$n=1,2,\ldots$}
      \State{$p^{(n)}\gets\prox_{r\phi_0}(z-rA_k^*\theta^{(n)})$}
      \State{$\mu^{(n)}\gets \ell_k(p^{(n)})$}
      \If{$n>1$}
        \State{$\gamma^{(n)}\gets\min\left\{\gamma_{\max},\max\left\{\gamma_{\min},\displaystyle{\frac{\lambda^{(n-1)}\|s^{(n-1)}\|_Y^2}{(\mu^{(n-1)}-\mu^{(n)},s^{(n-1)})_Y}}\right\}\right\}$}
      \EndIf
      \State{$s^{(n)}\gets\prox_{\gamma^{(n)}\phi_1^*}(\theta^{(n)}+\gamma^{(n)}\mu^{(n)})-\theta^{(n)}$}
      \State{$\lambda^{(n)}\gets\lambda_0$}
      \State{$\mathcal{Q}^{(n)}\gets(\mu^{(n)},s^{(n)})_Y+\phi_1^*(\theta^{(n)})-\phi_1^*(\theta^{(n)}+s^{(n)})$}
      \While{$V(\theta^{(n)}+\lambda^{(n)} s^{(n)}) < V^{(n)} + \alpha\lambda^{(n)}\mathcal{Q}^{(n)}$}
        \State{$\delta\gets-\displaystyle{\frac{[\lambda^{(n)}(\mu^{(n)},s^{(n)})_Y+\phi_1^*(\theta^{(n)})-\phi_1^*(\theta^{(n)}+\lambda^{(n)}s^{(n)})]}{2[d(\theta^{(n)}+\lambda^{(n)} s^{(n)})-d(\theta^{(n)})-\lambda^{(n)}(\mu^{(n)},s^{(n)})_Y]}}$}
        \If{$\delta\in[\sigma_1,\sigma_2]$}
          \State{$\lambda^{(n)}\gets\delta\lambda^{(n)}$}
        \Else
          \State{$\lambda^{(n)}\gets\beta\lambda^{(n)}$}
        \EndIf
      \EndWhile
      \State{$\theta^{(n+1)}\gets\theta^{(n)} + \lambda^{(n)}s^{(n)}$}
      \State{$V^{(n+1)}\gets(1-a^{(n)})V^{(n)}+a^{(n)}V(\theta^{(n+1)})$ for $a^{(n)}\in[a_{\min},1]$}
    \EndFor
  \end{algorithmic}
\end{algorithm}
The following proposition summarizes convergence results for
Algorithm~\ref{alg:prox-phik}.
\begin{proposition}
  Let $\{(p^{(n)},\theta^{(n)})\}$ denote the sequence of primal and dual
  iterates generated by Algorithm~\ref{alg:prox-phik}, then
  $\theta^{(n)}\rightharpoonup\bar\theta$, where $\bar\theta\in Y$ is a
  solution to the dual problem \eqref{eq:dualprox}. In addition, suppose one
  of the following conditions holds:
  \begin{enumerate}[label=(\alph*)]
  \item $A_k$ is a compact operator;
  \item $Y$ is finite dimensional;
  \item There exists
  $\varepsilon\in(0,\min\{1,(r\|A_k\|_{\cL(X,Y)}^2)^{-1}\})$ for
  which
  \[
    \varepsilon\le\gamma^{(n)}\le \frac{2}{r\|A_k\|_{\cL(X,Y)}^2}-\varepsilon
    \qquad\text{and}\qquad
    \varepsilon\le\lambda^{(n)}\le 1.
  \]
  \end{enumerate}
  Then, $p^{(n)}\to\prox_{r\psi_k}(z)=\prox_{r\phi_0}(z-rA_k^*\bar\theta)$.
\end{proposition}
\begin{proof}
  The proof of weak convergence follows arguments from \cite{bello2016weak} and
  \cite{de2023proximal}, which we have included in Appendix~\ref{ap:spg}.  In
  particular, $\{\theta^{(n)}\}$ are iterates generated by Algorithm~\ref{alg:spg}
  and so Theorem~\ref{th:spg-appendix} ensures weak convergence to a solution
  of the dual problem, which exists by
  \cite[Proposition~15.24(vii)~\&~Theorem~15.23]{bauschke2017convex} since
  $\phi_0$ is proper and $\phi_1$ is continuous. To prove the final claim, we
  first note that \cite[Proposition~19.5]{bauschke2017convex} ensures that
  $\bar p=\prox_{r\psi_k}(z)=\prox_{r\phi_0}(z-rA_k^*\bar\theta)$.  Consequently,
  if (a) holds, then $A_k^*$ is compact, which implies that
  $A_k^*\theta^{(n)}\to A_k^*\bar\theta$ and the result follows from the
  continuity of the proximity operator. Similarly, if (b) holds, then
  $\theta^{(n)}\to\bar\theta$ and again the result follows.  Finally,
  if (c) holds, then by firm nonexpansivity
  \cite[Proposition~4.4]{bauschke2017convex}, we have that
  \begin{align*}
    \|p^{(n)}-\bar p\|_X^2
    &\le r(-A_k^*\theta^{(n)}+A_k^*\bar\theta,\prox_{r\phi_0}(z-rA_k^*\theta^{(n)})-\bar p)_X \\
    &=r(\theta^{(n)}-\bar\theta,-A_k\prox_{r\phi_0}(z-rA_k^*\theta^{(n)})+A_k\bar p)_Y \\
    &=r(\theta^{(n)}-\bar\theta,\nabla d(\bar\theta)-\nabla d(\theta^{(n)}))_Y \\
    &\le r\|\theta^{(n)}-\bar\theta\|_Y \|\nabla d(\bar\theta)-\nabla d(\theta^{(n)})\|_Y.
  \end{align*}
  Since $\{\theta^{(n)}\}$ is weakly convergent,
  $\{\|\theta^{(n)}-\bar\theta\|_Y\}$ is bounded. In addition,
  \cite[Proposition~28.13]{bauschke2017convex} ensures that
  $\nabla d(\theta^{(n)})\to\nabla d(\bar\theta)$, proving the claim.
\end{proof}

Our subsequent convergence analysis assumes that the proximity operator of
$\psi_k$ is computed exactly.  However, Algorithm~\ref{alg:prox-phik} is only
guaranteed to produce an inexact evaluation, which motivates research on methods
that can leverage inexact proximity operators
\cite{salzo2012inexact,villa2013accelerated}. As we will see in the numerical
results, Algorithm~\ref{alg:prox-phik} typically produces sufficiently accurate
approximations with only modest effort.

\section{Convergence Theory}\label{sec:conv}
In this section, we prove convergence of Algorithm~\ref{alg:tr} assuming the
proximity operator $\prox_{t\psi_k}$ is computed exactly.  The technical
results are partitioned into two classes.  First, we prove global convergence of
Algorithm~\ref{alg:tr} in the sense that $\{h_k\}$ accumulates at zero.  These
results leverage traditional trust-region proof techniques like those used to
prove \cite[Theorem~2]{baraldi2022proximal}.  We then postulate additional
assumptions on the problem data that ensure the sequence of iterates
$\{x_k\}$ generated by Algorithm~\ref{alg:tr} accumulates at a stationary point
of \eqref{eq:optprob}.  For these results, we employ the following notation to
distinguish between quantities depending on approximations (i.e., $b_k$, $A_k$
and $g_k$) and those that do not:
\[
\begin{aligned}
  \hat{\ell}_k(x) &\coloneqq f_1'(x_k)(x-x_k)+f_1(x_k) \\
  \hat{\psi}_k(x) &\coloneqq \phi_0(x) + \phi_1(\hat{\ell}_k(x)) \\
  \hat{H}_k(t) &\coloneqq \tfrac{1}{t}\|x_k-\prox_{t\hat{\psi}_k}(x_k-t\nabla f_0(x_k))\|_X\\
  H_k(t) &\coloneqq \tfrac{1}{t}\|x_k-\prox_{t\psi_k}(x_k-tg_k)\|_X.
\end{aligned}
\]
Notice that $h_k=H_k(t_k)$ and we similarly denote $\hat{h}_k=\hat{H}_k(t_k)$.
We also use the notation $\xrightarrow{\cK}$ and $\xrightharpoonup{\cK}$ denote
strong and weak convergence, respectively, along the subsequence
$\cK\subseteq\mathbb{N}$.

\subsection{Global Convergence}

To prove global convergence of Algorithm~\ref{alg:tr}, we extend the theory
in \cite{baraldi2022proximal}.  In the subsequent lemma, we update 
\cite[Lemma~8]{baraldi2022proximal} to account for the more general nonsmooth
term $\psi_k$.

\begin{lemma}\label{L:success}
  Fix $k \ge K_\eta$ with $K_\eta$ defined in \eqref{eq:K-eta} and let
  Assumption~\ref{as:data} hold.  If $h_k > 0$ and
  \begin{equation}
    (1+\|B_k\|_{\cL(X)})\Delta_k \le \kappa_{\rm vs} h_k
  \end{equation}
  with $\kappa_{\rm vs}\in(0,1)$ defined by
  \[
    \kappa_{\rm vs}\coloneqq\frac{\kappa_{\rm fcd}(1-\eta_2-\eta)}{\max\{\kappa_{\rm fcd},
      \tfrac{1}{2}\kappa_{\rm rad}^2,
      \tfrac{M_{f_0}}{2}\kappa_{\rm rad}^2+\kappa_{\rm grad}\kappa_{\rm rad}
    + M_{\phi_1}(\tfrac{M_{f_1}}{2}\kappa_{\rm rad}^2
    + \kappa_{\rm jac}\kappa_{\rm rad} + 2\kappa_{\rm val})\}},
  \]
  then $\rho_k \ge \eta_2$ and $\Delta_{k+1}\ge\Delta_k$.
\end{lemma}
\begin{proof}
We first bound the difference between $\pred$ and $\ared$.  To this end,
let $s_k=x_k^+-x_k$ and recall that
\begin{align}
  \pred-\ared =&\, (q_k(x_k)+\phi_1(b_k)+\phi_0(x_k)-q_k(x_k^+)-\phi_1(\ell_k(x_k^+))-\phi_0(x_k^+)) \nonumber \\
  &-(f_0(x_k)+\phi_1(f_1(x_k))+\phi_0(x_k)-f_0(x_k^+)-\phi_1(f_1(x_k^+))-\phi_0(x_k^+)) \nonumber \\
  =&\,(q_k(x_k)-q_k(x_k^+))-(f_0(x_k)-f_0(x_k^+)) \nonumber \\
  &+(\phi_1(b_k)-\phi_1(\ell_k(x_k^+))-(\phi_1(f_1(x_k))-\phi_1(f_1(x_k^+))) \nonumber.
\end{align}
We individually bound the smooth and nonsmooth terms. For the smooth term,
we have that
\begin{align}
  (q_k(x_k)-q_k(x_k^+))&-(f_0(x_k)-f_0(x_k^+)) \nonumber \\
    &= -\tfrac{1}{2}(B_ks_k,s_k)_X - (g_k,s_k)_X
       + \int_0^1 (\nabla f_0(x_k+ts_k),s_k)_X\,\mathrm{d}t \nonumber \\
    &\le \tfrac{1}{2}\|B_k\|_{\mathcal{L}(X)}\|s_k\|_X^2 + \tfrac{M_{f_0}}{2}\|s_k\|_X^2
       + \kappa_{\rm grad}\|s_k\|_X\Delta_k \nonumber \\
    &\le (\tfrac{1}{2}\kappa_{\rm rad}^2\|B_k\|_{\cL(X)} + \tfrac{M_{f_0}}{2}\kappa_{\rm rad}^2
       + \kappa_{\rm grad}\kappa_{\rm rad})\Delta_k^2, \nonumber
\end{align}
where we used Assumption~\ref{as:data}, Condition~\ref{as:inexactGrad},
and \eqref{eq:trcon}.  For the nonsmooth term, Condition~\ref{as:inexactGrad} and
\eqref{eq:trcon} ensure that
\begin{align}
  \phi_1(b_k)-\phi_1(f_1(x_k))
    \le M_{\phi_1}\|b_k-f_1(x_k)\|_Y
    \le M_{\phi_1}\kappa_{\rm val}\Delta_k^2 \nonumber
\end{align}
and we have that
\begin{equation} \label{eq:success1}
\begin{aligned}
  \phi_1(f_1(x_k^+))-\phi_1(\ell_k(x_k^+))
   =&\, (\phi_1(f_1(x_k^+))-\phi_1(f_1(x_k)+f_1'(x_k)s_k)) \\
    &+(\phi_1(f_1(x_k)+f_1'(x_k)s_k)-\phi_1(\ell_k(x_k^+))).
\end{aligned}
\end{equation}
Again using Condition~\ref{as:inexactGrad}, we can bound the second term on the
right-hand side of \eqref{eq:success1} as
\begin{align*}
  \phi_1(f_1(x_k)+f_1'(x_k)s_k)-\phi_1(\ell_k(x_k^+))
    &\le M_{\phi_1}(\|(f_1'(x_k)-A_k)s_k\|_Y+\|b_k-f_1(x_k)\|_Y) \\
    &\le M_{\phi_1}(\kappa_{\rm jac}\|s_k\|_X\Delta_k+\kappa_{\rm val}\Delta_k^2) \\
    &\le M_{\phi_1}(\kappa_{\rm jac}\kappa_{\rm rad}+\kappa_{\rm val})\Delta_k^2,
\end{align*}
Finally, to bound the first term on the right-hand side of \eqref{eq:success1},
we first recall that for any $\theta\in\partial\phi_1(f_1(x_k^+))$ we have that
$\|\theta\|_Y\le M_{\phi_1}$ and
\[
  \phi_1(f_1(x_k^+)) - \phi_1(f_1(x_k)+f_1'(x_k)s_k)) \le (\theta,f_1(x_k^+)-f_1(x_k)-f_1'(x_k)s_k))_Y.
\]
Consequently, Assumption~\ref{as:data} ensures that
\begin{align*}
  \phi_1(f_1(x_k^+)) - \phi_1(f_1(x_k)+f_1'(x_k)s_k))
   &\le \int_0^1 (\theta,(f_1'(x_k+ts_k)-f_1'(x_k))s_k)_Y\,\mathrm{d}t \\
   &\le \tfrac{1}{2} M_{\phi_1} M_{f_1} \|s_k\|_X^2 \\
   &\le \tfrac{1}{2} M_{\phi_1} M_{f_1} \kappa_{\rm rad}^2 \Delta_k^2,
\end{align*}
where we have used the fundamental theorem of calculus applied to the map
$t\mapsto (\theta, f_1(x_k+ts_k))_Y$ for fixed
$\theta\in\partial\phi_1(f_1(x_k^+))$ to arrive at the first equality.
Combining these bounds yields
\begin{equation}\label{eq:success2}
  \pred-\ared \le \frac{\kappa_{\rm fcd}(1-\eta_2-\eta)}{\kappa_{\rm vs}}(1+\|B_k\|_{\cL(X)})\Delta_k^2,
\end{equation}
where we have used H\"{o}lder's inequality.  Combining \eqref{eq:success2} and
\eqref{eq:fcd}, we obtain
\[
  1-\rho_k^* \le \frac{\kappa_{\rm fcd}(1-\eta_2-\eta)}{\kappa_{\rm vs}}
  \frac{(1+\|B_k\|_{\cL(X)})\Delta_k^2}{\kappa_{\rm fcd}h_k\min\left\{\frac{h_k}{1+\|B_k\|_{\cL(X)}},\Delta_k\right\}}
  \le (1-\eta_2-\eta).
\]
Consequently, we have that $\rho_k^* \ge \eta_2 + \eta$ and $\rho_k \ge \eta_2$
as was to be shown.
\end{proof}

Replacing \cite[Lemma~8]{baraldi2022proximal} with Lemma~\ref{L:success} in the
proof of \cite[Theorem~3]{baraldi2022proximal} enables us to prove the
convergence of Algorithm~\ref{alg:tr}.

\begin{theorem}\label{T:conv}
  Let $\{x_k\}$ be the sequence of iterates generated by Algorithm~\ref{alg:tr}.
  If Assumption~\ref{as:data} holds and if
  \[
    \sum_{k=1}^\infty (1+\displaystyle{\max_{j=1,\ldots,k}}\|B_j\|_{\cL(X)})^{-1} = +\infty,
  \]
  then
  \[
    \liminf_{k\to\infty}\,h_k = 0.
  \]
  In addition, if there exists $t_{\max}>0$ satisfying $t_k\le t_{\max}$ for
  all $k$, then
  \[
    \liminf_{k\to\infty}\,H_k(t) = 0 \quad\forall\, t>0.
  \]
\end{theorem}
\begin{proof}
  The proof of this result is nearly identical to the proof of
  \cite[Theorem~3,~Equation~(40)]{baraldi2022proximal} with
  \cite[Lemma~8]{baraldi2022proximal} replaced by Lemma~\ref{L:success}.
  Finally, the monotonicity of $t\mapsto H_k(t)$ (cf.\
  \cite[Lemma~2]{baraldi2022proximal}) and the upper bound on $t_k$ ensure that
  $H_k(t) \le \max\{1,t_{\max}/t\}h_k$ for all $t>0$, proving the final result.
\end{proof}

In our first corollary to Theorem~\ref{T:conv}, we relate the convergence
of $h_k$ and $H_k(t)$ with $\hat{h}_k$ and $\hat{H}_k(t)$.

\begin{corollary}\label{C:conv1}
  Let the assumptions of Theorem~\ref{T:conv} hold. If there exist
  $0<t_{\min}\le t_{\max}$ such that $t_k\in[t_{\min},t_{\max}]$ for all $k$
  then
  \[
    \liminf_{k\to\infty}\; \hat{h}_k = 0
    \qquad\text{and}\qquad
    \liminf_{k\to\infty}\; \hat{H}_k(t) = 0 \quad\forall\, t>0.
  \]
\end{corollary}
\begin{proof}
  We first prove a bound between the proximity
  operators associated with $\psi_k$ and $\hat{\psi}_k$.  Let
  $p_k=\prox_{t_k\psi_k}(x_k-t_k\nabla f_0(x_k))$ and
  $\hat{p}_k=\prox_{t_k\hat{\psi}_k}(x_k-t_k\nabla f_0(x_k))$.  Applying
  Lemma~\ref{L:proxDiff} with $\Psi_1=\psi_k$ and $\Psi_2=\hat{\psi}_k$
  and noting that
  \[
    \left\{
    \begin{aligned}
    &\|D_1-D_2\|_{\cL(X,Y)}=\|A_k-f_1'(x_k)\|_{\cL(X,Y)}\le \kappa_{\rm jac} h_k \\
    &\|p_1-u_1\|_X=\|p_k-x_k\|_X\le \kappa_{\rm grad}t_kh_k + t_kh_k \le t_{\max}(\kappa_{\rm grad}+1)h_k \\
    &\|d_1-d_2-D_2(u_1-u_2)\|_Y=\|b_k-f_1(x_k)\|_Y\le \kappa_{\rm val}h_k,
    \end{aligned}
    \right.
  \]
  we arrive at the bound
  \[
    \frac{1}{t_k}\frac{\|p_k-\hat{p}_k\|_X^2}{\max\{1,\|p_k-\hat{p}_k\|_X\}}
    \le 2M_{\phi_1}(\kappa_{\rm jac}(1+2t_{\max}(\kappa_{\rm grad}+1)h_k)+2\kappa_{\rm val})h_k.
  \]
  Let $\mathcal{K}\subseteq\mathbb{N}$ denote the indices of a subsequence on
  which $h_k\xrightarrow{\mathcal{K}}0$.
  Notice that the above bound and the boundedness of $t_k$ imply that
  \[
    \frac{1}{t_k}\|p_k-\hat{p}_k\|_X \xrightarrow{\mathcal{K}} 0.
  \]
  The nonexpansivity of the proximity operator and
  Condition~\ref{as:inexactGrad} then ensure that
  \[
  \begin{aligned}
    \hat h_k
    =&\, \tfrac{1}{t_k}\|x_k - \prox_{t_k\hat{\psi}_k}(x_k-t_k\nabla f_0(x_k))\|_X \\
    \le&\,  \tfrac{1}{t_k}\|x_k - \prox_{t_k\psi_k}(x_k-t_k\nabla f_0(x_k))\|_X
          + \tfrac{1}{t_k}\|p_k-\hat{p}_k\|_X \\
    \le&\, (\kappa_{\rm grad}+1)h_k + \tfrac{1}{t_k}\|p_k-\hat{p}_k\|_X
  \end{aligned}
  \]
  and therefore $\hat h_k\xrightarrow{\mathcal{K}}0$, proving that the lower
  limit of $\hat h_k$ is zero. The final result follows from the monotonicity
  of $t\mapsto \hat H_k(t)$ \cite[Lemma~2]{baraldi2022proximal}. In particular,
  $\hat H_k(t) \le \hat h_k$ for all $t\ge t_{\max}$ and for any $t\le t_{\max}$,
  $\hat H_k(t) \le \frac{t_{\max}}{t} \hat H_k(t_{\max}) \le \frac{t_{\max}}{t}\hat h_k$.
\end{proof}

Under additional assumptions, we can strengthen the lower limit in
Theorem~\ref{T:conv} to a limit.

\begin{theorem}\label{T:conv-lim}
  Let the assumptions of Corollary~\ref{C:conv1} hold. If there exists
  $\kappa_{\rm curv}>0$ satisfying $\|B_k\|_{\cL(X)}\le\kappa_{\rm curv}$ for
  all $k\in\mathbb{N}$, then
  \[
    \lim_{k\to\infty} h_k = \lim_{k\to\infty} H_k(t)
    = \lim_{k\to\infty} \hat{H}_k(t) = 0 \quad\forall\, t>0.
  \]
\end{theorem}
\begin{proof}
  The proof of this is similar to the proof of
  \cite[Theorem~1]{baraldi2024local} with modifications to account for the
  more general objective function.  Let
  $\mathcal{S}\coloneqq\{k\in\mathbb{N}\,\vert\,\rho_k\ge\eta_1\}$ denote the
  set of indices corresponding to successful iterations.  If $\mathcal{S}$ is
  finite, then $x_k=\bar{x}$ for $k>\max\mathcal{S}$ and $\Delta_k\to 0$.
  Moreover, by Lemma~\ref{L:success} and the boundedness of $\|B_k\|_{\cL(X)}$,
  we have that $h_k \le [(1+\kappa_{\rm curv})/\kappa_{\rm vs}]\Delta_k$ for
  $k>\max\{\max\mathcal{S},K_\eta\}$ and so $h_k\to 0$.  The result then
  follows by bounding the quantities $H_k(t)$ and $\hat H_k(t)$ by $h_k$ for
  $t>0$ using the monotonicity properties of the proximal gradient from
  \cite[Lemma~2]{baraldi2022proximal}.

  Now suppose that $\mathcal{S}$ is infinite and set $t=t_{\max}$. To arrive at
  a contradiction, we assume that there exists $\epsilon>0$ and a subsequence
  $\{k_i\}\subseteq\mathcal{S}$ satisfying
  \begin{equation}\label{eq:conv-lim-1}
    h_{k_i}\ge H_{k_i}(t) \ge 2\epsilon > 0 \quad\forall\,i\in\mathbb{N}.
  \end{equation}
  By Theorem~\ref{T:conv}, for each $i$ there exists $\ell>k_i$ for
  which $H_\ell(t)<\epsilon$.  Let $\ell_i>k_i$ denote the first such index
  and define
  \[
    \mathcal{S}_0\coloneqq\{j\in\mathcal{S}\,\vert\, \exists \, i\in\mathbb{N}\quad\text{such that}\quad k_i \le j < \ell_i\}.
  \]
  In addition, let $j_i$ denote the largest index in $\mathcal{S}_0$ for which
  $k_i \le j_i < \ell_i$ for each $i\in\mathbb{N}$.
  Note that for all $j\in\mathcal{S}_0$, we have that $H_j(t)\ge\epsilon$.
  By \eqref{eq:K-eta}, there exists $K_\eta\in\mathbb{N}$
  such that $|\rho_j^*-\rho_j|\le\eta$ for all $j\ge K_\eta$ and therefore
  $$J(x_j)-J(x_{j+1})\ge {\rm cred}_j-\eta{\rm pred}_j \ge (\eta_1-\eta){\rm pred}_j\ge0$$
  for all $j\in\mathcal{S}_0$ with $j\ge K_\eta$, i.e.,
  $\{J(x_k)\}_{k\ge K_\eta}$ is monotonically decreasing.
  This, the boundedness of $\{\|B_k\|_{\cL(X)}\}$ and \eqref{eq:fcd} ensure that
  \[
    J(x_j)-J(x_{j+1}) \ge (\eta_1-\eta)\kappa_{\rm fcd}\epsilon\min\left\{\frac{\epsilon}{1+\kappa_{\rm curv}},\Delta_j\right\}
  \]
  for all $j\in\mathcal{S}$ with $j\ge K_\eta$. Since
  $\{J(x_k)\}_{k\ge K_\eta}$ is monotonically decreasing and bounded below, we
  have that the left-hand side converges to zero and hence so does $\Delta_j$.
  Consequently, for sufficiently large $j\in\mathcal{S}_0$, we obtain
  \[
    \Delta_j \le \frac{J(x_j)-J(x_{j+1})}{\kappa_{\rm fcd}\epsilon(\eta_1-\eta)}.
  \]
  For sufficiently large $i$, this and the triangle inequality ensure that
  \[
    \|x_{k_i}-x_{\ell_i}\|_X \le \sum_{\substack{j=k_i \\ j\in\mathcal{S}_0}}^{\ell_i-1} \|x_j-x_{j+1}\|_X
    \le \kappa_{\rm rad}\sum_{\substack{j=k_i \\ j\in\mathcal{S}_0}}^{\ell_i-1}\Delta_j
    \le \frac{\kappa_{\rm rad}[J(x_{k_i})-J(x_{\ell_i})]}{\kappa_{\rm fcd}\epsilon(\eta_1-\eta)}.
  \]
  The right-hand side converges to zero and therefore so does the left-hand
  side as $i\to\infty$.  Lipschitz continuity then ensures that
  $\|\nabla f_0(x_{k_i})-\nabla f_0(x_{\ell_i})\|_X$ and
  $\|f_1'(x_{k_1})-f_1'(x_{\ell_i})\|_{\cL(X,Y)}$ tend to zero as $i\to\infty$.
  Now, the reverse triangle inequality yields
  \[
  \begin{aligned}
    \epsilon &\le |H_{k_i}(t)-H_{\ell_i}(t)| \\
    &\le \frac{1}{t}\|x_{k_i}-x_{\ell_i}\|_X + \frac{1}{t}\|\prox_{t\psi_{k_i}}(x_{k_i}-tg_{k_i})-\prox_{t\psi_{\ell_i}}(x_{\ell_i}-tg_{\ell_i})\|_X.
  \end{aligned}
  \]
  Nonexpansivity of the proximity operator then implies that
  \begin{align}
    \|\prox_{t\psi_{k_i}}(x_{k_i}-tg_{k_i})&-\prox_{t\psi_{\ell_i}}(x_{\ell_i}-tg_{\ell_i})\|_X \nonumber\\
    \le& \|x_{k_i}-x_{\ell+i}\|_X + t\|g_{k_i}-g_{\ell_i}\|_X \nonumber\\
       &+ \|\prox_{t\psi_{k_i}}(x_{k_i}-tg_{k_i})-\prox_{t\psi_{\ell_i}}(x_{k_i}-tg_{k_i})\|_X.\label{eq:conv-lim-2}
  \end{align}
  Owing to Condition~\ref{as:inexactGrad} and Lipschitz continuity, we have
  that
  \[
  \begin{aligned}
    \|g_{k_i}-g_{\ell_i}\|_X &\le \kappa_{\rm grad}(\Delta_{k_i} + \Delta_{\ell_i}) + M_{f_0}\|x_{k_i}-x_{\ell_i}\|_X \\
    &\le \kappa_{\rm grad}(\Delta_{k_i} + \gamma_3\Delta_{j_i}) + M_{f_0}\|x_{k_i}-x_{\ell_i}\|_X,
  \end{aligned}
  \]
  where the final inequality follows from the trust-radius update mechanism
  in Algorithm~\ref{alg:tr} (i.e., $\Delta_{\ell_i}\le\gamma_3\Delta_{j_i}$).
  It remains to bound the final term on the right-hand side of
  \eqref{eq:conv-lim-2}. For this, we employ Lemma~\ref{L:proxDiff}
  with $\Psi_1=\psi_{\ell_i}$ and $\Psi_2=\psi_{k_i}$.  In this setting,
  we have that
  \[
    \left\{
    \begin{aligned}
    &\|D_1-D_2\|_{\cL(X,Y)}=\|A_{\ell_i}-A_{k_i}\|_{\cL(X,Y)} \\
    &\|p_1-u_1\|_X=\|\prox_{t\psi_{\ell_i}}(x_{k_i}-tg_{k_i})-x_{\ell_i}\|_X \le \|x_{k_i}-x_{\ell_i}\|_X+t\|g_{k_i}-g_{\ell_i}\|_X + tH_{\ell_i}(t) \\
    &\|d_1-d_2-D_2(u_1-u_2)\|_Y=\|b_{\ell_i}-b_{k_i}-A_{k_i}(x_{\ell_i}-x_{k_i})\|_Y.
    \end{aligned}
    \right.
  \]
  Note that $\|A_{\ell_i}-A_{k_i}\|_{\cL(X,Y)}$ tends to zero because of
  Condition~\ref{as:inexactGrad} and the fact that $\Delta_{k_i}$ and
  $\|x_{k_i}-x_{\ell_i}\|_X$ tend to zero.  That is,
  \[
    \begin{aligned}
    \|A_{\ell_i}-A_{k_i}\|_{\cL(X,Y)} \le&\, \|A_{k_i}-f_1'(x_{k_i})\|_{\cL(X,Y)} + \|A_{\ell_i}-f_1'(x_{\ell_i})\|_{\cL(X,Y)} \\
       &+ \|f_1'(x_{k_i})-f_1'(x_{\ell_i})\|_{\cL(X,Y)} \\
    \le&\, \kappa_{\rm jac}(\Delta_{k_i}+\Delta_{\ell_i}) + M_{f_1} \|x_{k_i}-x_{\ell_i}\|_X \\
    \le&\, \kappa_{\rm jac}(\Delta_{k_i}+\gamma_3\Delta_{j_i}) + M_{f_1} \|x_{k_i}-x_{\ell_i}\|_X.
    \end{aligned}
  \]
  Using similar arguments combined with the mean value theorem, we can bound
  \[
    \begin{aligned}
    \|b_{\ell_i}-&b_{k_i}-A_{k_i}(x_{\ell_i}-x_{k_i})\|_Y \\
    \le&\, \|b_{k_i}-f_1(x_{k_i})\|_Y + \|b_{\ell_i}-f_1(x_{\ell_i})\|_{Y} + \|A_{k_i}-f_1'(x_{k_i})\|_{\cL(X,Y)}\|x_{k_i}-x_{\ell_i}\|_X \\
       &+ \|f_1(x_{\ell_i})-f_1(x_{k_i})-f_1'(x_{k_i})(x_{\ell_i}-x_{k_i})\|_Y \\
    \le&\, \kappa_{\rm val}(\Delta_{k_i}^2+\Delta_{\ell_i}^2)+\kappa_{\rm jac}\Delta_{k_i}\|x_{k_i}-x_{\ell_i}\|_X
          +\tfrac{1}{2}M_{f_1}\|x_{k_i}-x_{\ell_i}\|_X^2 \\
    \le&\, \kappa_{\rm val}(\Delta_{k_i}^2+\gamma_3^2\Delta_{j_i}^2)+\kappa_{\rm jac}\Delta_{k_i}\|x_{k_i}-x_{\ell_i}\|_X
          +\tfrac{1}{2}M_{f_1}\|x_{k_i}-x_{\ell_i}\|_X^2.
    \end{aligned}
  \]
  Finally, there exists $c>0$, independent of $k$, such that
  $\|\prox_{t\psi_{\ell_i}}(x_{k_i}-tg_{k_i})-x_{\ell_i}\|_X\le c$ for all $k$ since
  $\|x_{k_i}-x_{\ell_i}\|_X$ and $\|g_{k_i}-g_{\ell_i}\|_X$ converge to zero, and
  $H_{\ell_i}(t)\le\varepsilon$.
  Defining
  \[
    \begin{aligned}
    p_i&=\prox_{t\psi_{k_i}}(x_{k_i}-tg_{k_i}), \qquad
    \tilde{p}_i=\prox_{t\psi_{\ell_i}}(x_{k_i}-tg_{k_i}), \\
    \mu_i&=M_{\phi_1}(\kappa_{\rm jac}(\Delta_{k_i}+\gamma_3\Delta_{j_i}) + M_{f_1} \|x_{k_i}-x_{\ell_i}\|_X) \quad\text{and}\quad\\
    \nu_i&=M_{\phi_1}(\kappa_{\rm val}(\Delta_{k_i}^2+\gamma_3^2\Delta_{j_i}^2)+\kappa_{\rm jac}\Delta_{k_i}\|x_{k_i}-x_{\ell_i}\|_X
      +\tfrac{1}{2}M_{f_1}\|x_{k_i}-x_{\ell_i}\|_X^2),
    \end{aligned}
  \]
  we arrive at the bound
  \[
    \frac{1}{t}\frac{\|p_i-\tilde{p}_i\|_X^2}{\max\{1,\|p_i-\tilde{p}_i\|_X\}}
    \le 2((1+2tc)\mu_i+\nu_i).
  \]
  Notice that $\mu_i\to0$ and $\nu_i\to0$.  Therefore,
  $\|p_i-\tilde{p}_i\|_X\to0$. As a result, $|H_{k_i}(t)-H_{\ell_i}(t)|$
  converges to zero, arriving at a contradiction.  Hence, no such subsequence
  exists, implying that $H_k(t)\to 0$ and the desired result follows as before.
\end{proof}

\subsection{Convergence to Stationary Points}

To prove convergence to a stationary point, we leverage results of the previous
subsection.  However, we require additional regularity of the problem data.
The principle challenge is in demonstrating convergence of the proximity
operators of $\psi_k$ to the proximity operator of $\psi_{\bar{x}}$, where
$\bar{x}$ is a stationary point of \eqref{eq:optprob}.  The following
proposition postulates conditions on $F$ for which this property is valid.

\begin{proposition}\label{P:mosco}
  Let $\{x_k\}$ be the sequence of iterates generated by Algorithm~\ref{alg:tr}
  and suppose Assumption~\ref{as:data} holds.  Moreover, suppose there exists
  $\bar{x}\in X$ and a subsequence index set $\cK\subseteq\mathbb{N}$ such that
  $\{x_k\}_{\cK}$ is bounded, $h_k\xrightarrow{\cK}0$,
  $f_1'(x_k)\xrightarrow{\cK}f_1'(\bar{x})$, and
  \[
    f_1(x_k)-f_1(\bar x)-f_1'(\bar x)(x_k-\bar x)\xrightarrow{\cK}0.
  \]
  Then,
  \begin{equation}\label{eq:prox-conv}
    \prox_{t\psi_k}(x) \xrightarrow{\cK} \prox_{t\psi_{\bar{x}}}(x) \quad\forall\,x\in X, \;\;t>0.
  \end{equation}
\end{proposition}
\begin{proof}
  Fix $x\in X$ and $t>0$. Lemma~\ref{L:proxDiff} with $\Psi_1=\psi_{k}$ and
  $\Psi_2=\psi_{\bar x}$ (i.e., $D_1=A_k$, $d_1=b_k$, $u_1=x_k$,
  $D_2=f_1'(\bar x)$, $d_2=f_1(\bar x)$ and $u_2=\bar x$), combined with
  Condition~\ref{as:inexactGrad}, yields
  \[
    \left\{
    \begin{aligned}
    &\|D_1-D_2\|_{\cL(X,Y)}\le \kappa_{\rm jac} h_k + \|f_1'(x_k)-f_1'(\bar{x})\|_{\cL(X,Y)} \\
    &\|p_1-u_1\|_X=\|\prox_{t\psi_k}(x)-x_k\|_X \\
    &\|d_1-d_2-D_2(u_1-u_2)\|_Y\le \kappa_{\rm val}h_k + \|f_1(x_k)-f_1(\bar x)-f_1'(\bar x)(x_k-\bar{x})\|_Y.
    \end{aligned}
    \right.
  \]
  If $\{\|\prox_{t\psi_k}(x)-x_k\|_X\}_{\cK}$ is bounded, then our assumptions
  ensure that
  \[
    \lim_{\cK\ni k\to 0}\frac{\|\prox_{t\psi_k}(x)-\prox_{t\psi_{\bar{x}}}(x)\|_X^2}{\max\{1,\|\prox_{t\psi_k}(x)-\prox_{t\psi_{\bar{x}}}(x)\|_X\}}
    = 0,
  \]
  and the desired result holds.  As such, we will prove that
  $\{\|\prox_{t\psi_k}(x)-x_k\|_X\}_{\cK}$ is bounded.  To this end, we note
  that
  \[
    \begin{aligned}
    |(\psi_k(z_1)-\psi_{\bar x}(z_1))&-(\psi_k(z_2)-\psi_{\bar x}(z_2))|
    \le |(\phi_1(\ell_k(z_1))-\phi_1(\ell_k(z_2))| \\
       &+|(\phi_1(f_1(\bar x)+f_1'(\bar x)(z_1-\bar x))-\phi_1(f_1(\bar x)+f_1'(\bar x)(z_2-\bar x))| \\
       &\le M_{\phi_1}(\|A_k\|_{\mathcal{L}(X,Y)} + \|f_1'(\bar x)\|_{\cL(X,Y)})\|z_1-z_2\|_X
    \end{aligned}
  \]
  for all $z_1,\,z_2\in X$ and therefore
  \cite[Proposition~4.32]{JFBonnans_AShapiro_2000} ensures that
  \[
    \|\prox_{t\psi_k}(x)-\prox_{t\psi_{\bar x}}(x)\|_X \le tM_{\phi_1}(\|A_k\|_{\cL(X,Y)} + \|f_1'(\bar x)\|_{\cL(X,Y)}).
  \]
  Applying Condition~\ref{as:inexactGrad} yields
  \[
    \|A_k\|_{\cL(X,Y)} \le \kappa_{\rm jac}h_k
      + \|f_1'(x_k)-f_1'(\bar x)\|_{\cL(X,Y)}
      + \|f_1'(\bar x)\|_{\cL(X,Y)},
  \]
  which is bounded for $k\in\cK$ since $f_1'(x_k)\xrightarrow{\cK}f_1'(\bar x)$ and
  $h_k\xrightarrow{\cK}0$.  Therefore, we have that
  \[
    \|\prox_{t\psi_k}(x)\|_X \le \|\prox_{t\psi_k}(x)-\prox_{t\psi_{\bar x}}(x)\|_X + \|\prox_{t\psi_{\bar x}}(x)\|_X
  \]
  is bounded for $k\in\cK$, proving the claim.
\end{proof}

In the subsequent corollaries, we provide conditions under which the
assumptions of Proposition~\ref{P:mosco} hold.  Our first corollary
is particularly useful when $X$ is finite dimensional as it assumes
the existence of a strongly converging subsequence, while the second
corollary is tailored to PDE-constrained optimization applications
in which $F$ and $F'$ are completely continuous.  Recall that completely
continuous functions map weakly converging sequences into strongly
converging sequences \cite{emelyanov2011homotopy}.

\begin{corollary}\label{C:mosco-strong}
  Let $\{x_k\}$ be the sequence of iterates generated by Algorithm~\ref{alg:tr}
  and suppose Assumption~\ref{as:data} holds.  Moreover, suppose that
  $\bar{x}\in X$ and $\cK$ are such that $h_k\xrightarrow{\cK}0$ and
  $x_k\xrightarrow{\cK}\bar{x}$. Then \eqref{eq:prox-conv} holds.
\end{corollary}
\begin{proof}
  Assumption~\ref{as:data} and strong convergence ensure that
  $\{x_k\}_{\cK}$ is bounded,
  \begin{align*}
    f_1(x_k)\xrightarrow{\cK}f_1(\bar{x}) \quad\text{and}\quad
    f_1'(x_k)\xrightarrow{\cK}f_1'(\bar{x}).
  \end{align*}
  The result then follows from Proposition~\ref{P:mosco}.
\end{proof}
\begin{corollary}\label{C:mosco-weak}
  Let $\{x_k\}$ be the sequence of iterates generated by Algorithm~\ref{alg:tr}
  and suppose Assumption~\ref{as:data} holds.  Moreover, suppose that
  $\bar{x}\in X$ and $\cK$ are such that $h_k\xrightarrow{\cK}0$ and
  $x_k\xrightharpoonup{\cK}\bar{x}$. Finally, assume that $f_1$ and $f_1'$ are
  completely continuous. Then \eqref{eq:prox-conv} holds.
\end{corollary}
\begin{proof}
  Weak convergence ensures that $\{x_k\}_{\cK}$ is bounded and the assumed
  complete continuity ensures that
  \[
    f_1(x_k)\xrightarrow{\cK}f_1(\bar{x})\quad\text{and}\quad
    f_1'(x_k)\xrightarrow{\cK}f_1'(\bar{x}).
  \]
  Moreover, $f_1'(\bar{x})$ is a completely continuous operator
  \cite[Theorem~1.5.1]{emelyanov2011homotopy} and therefore,
  $f_1'(\bar x)(x_k-\bar x)\xrightarrow{\cK}0$.
  The result then follows from Proposition~\ref{P:mosco}.
\end{proof}

Combining Corollaries~\ref{C:mosco-strong}~and~\ref{C:mosco-weak} with
Theorem~\ref{T:conv} enables us to prove convergence to stationary points of
\eqref{eq:optprob}.  Note that a consequence of Theorem~\ref{T:conv} is that
there exists a subsequence of the stationarity measures $\{h_{k_j}\}$ that
converges to $0$.  Our first result demonstrates that strong accumulation
points are stationary.

\begin{theorem}\label{T:dual-strong}
  Let $\{x_k\}$ be the sequence of iterates generated by Algorithm~\ref{alg:tr}
  and let the assumptions of Theorem~\ref{T:conv} hold. Moreover, suppose
  there exists $\bar{x}\in X$ and a subsequence index set
  $\cK\subseteq\mathbb{N}$ such that
  $h_k\xrightarrow{\cK}0$ and $x_k\xrightarrow{\cK}\bar{x}$.  If there exists
  $t_{\max}>0$ with $t_k\le t_{\max}$ for all $k$, then $\bar{x}$ is a
  stationary point of \eqref{eq:optprob}.  In particular, \eqref{eq:statpt}
  holds for all $t>0$.
\end{theorem}
\begin{proof}
  By Corollary~\ref{C:mosco-strong}, \eqref{eq:prox-conv} holds. Using
  Condition~\ref{as:inexactGrad}, we have that
  \[
    x_k-tg_k\xrightarrow{\cK}\bar{x}-t\nabla f_0(\bar{x})
  \]
  for fixed $t>0$ and so the nonexpansivity of the proximity
  operator ensures that
  \[
    \begin{aligned}
    \|\prox_{t\psi_k}(x_k-tg_k)-&\prox_{t\psi_{\bar x}}(\bar{x}-t\nabla f_0(\bar{x}))\|_X \\
    \le& \|(x_k-tg_k)-(\bar{x}-t\nabla f_0(\bar{x}))\|_X \\
       &+\|\prox_{t\psi_k}(\bar{x}-t\nabla f_0(\bar{x}))-\prox_{t\psi_{\bar x}}(\bar{x}-t\nabla f_0(\bar{x}))\|_X.
    \end{aligned}
  \]
  In particular, we have that
  \[
    \prox_{t\psi_k}(x_k-tg_k)\xrightarrow{\cK}\prox_{t\psi_{\bar x}}(\bar{x}-t\nabla f_0(\bar{x})).
  \]
  Now, the monotonicity properties of the proximal gradient
  \cite[Lemma~2]{baraldi2022proximal} yield $H_k(t)\le \max\{1,t_{\max}/t\}h_k$,
  which combined with the assumption that $h_k\xrightarrow{\cK}0$, ensures that
  \[
    H_k(t)=\frac{1}{t}\|\prox_{t\psi_k}(x_k-tg_k)-x_k\|_X\xrightarrow{\cK} 0.
  \]
  Therefore, $\bar{x}=\prox_{t\psi_{\bar x}}(\bar{x}-t\nabla f_0(\bar{x}))$.
  The result then follows from Theorem~\ref{T:statpt}.
\end{proof}

Theorem~\ref{T:dual-strong} is a useful result when $X$ is finite dimensional.
However, strong convergence is difficult to guarantee in infinite-dimensional
Hilbert spaces where bounded sequences are only guaranteed to admit weak
accumulation points.  As demonstrated by the following theorem, this setting
requires additional assumptions on $f_0$ and $f_1$.

\begin{theorem}\label{T:strong-convex}
  Let $\{x_k\}$ be the sequence of iterates generated by Algorithm~\ref{alg:tr}
  and let the assumptions of Corollary~\ref{C:conv1} hold.  Moreover, suppose
  there exists $\bar{x}\in X$ and a subsequence index set
  $\cK\subseteq\mathbb{N}$ such
  that $h_k\xrightarrow{\cK}0$ and $x_k\xrightharpoonup{\cK}\bar{x}$.
  Finally, suppose that $f_0$ is strongly convex and $f_1$ and $f_1'$ are
  completely continuous. If $\bar{x}$ is a stationary point for
  \eqref{eq:optprob}, then $x_k\xrightarrow{\cK}\bar{x}$.
\end{theorem}
\begin{proof}
  Let $\omega>0$ denote the strong convexity modulus of $f_0$.
  To simplify the presentation, define
  \[
    G_k(x,t) \coloneqq \tfrac{1}{t}(x-\prox_{t\psi_k}(x - t\nabla f_0(x))).
  \]
  By \cite[Lemma~2]{baraldi2024local}, we have that $G_k(\cdot,t)$ is a
  strongly monotone operator for all $t\in(0,2\omega M_{f_0}^{-2})$, which
  yields
  \[
    (\omega-\tfrac{1}{2}tM_{f_0}^{2})\|x_k-\bar{x}\|_X^2
    \le (G_k(x_k,t)-G_k(\bar{x},t),x_k-\bar{x})_X.
  \]
  Applying the Cauchy-Schwarz and triangle inequalities, we obtain
  \[
    (\omega-\tfrac{1}{2}tM_{f_0}^{2})\|x_k-\bar{x}\|_X \le \|G_k(x_k,t)\|_X + \|G_k(\bar{x},t)\|_X.
  \]
  By the monotonicity of the proximal gradient
  \cite[Lemma~2]{baraldi2022proximal}, the assumption that
  $h_k\xrightarrow{\cK}0$ and Condition~\ref{as:inexactGrad}, we have that
  \[
    \lim_{\cK\ni k\to\infty} \|G_k(x_k,t)\|_X \le \lim_{\cK\ni k\to\infty} \kappa_{\rm grad}h_k + H_k(t) = 0
  \]
  and by Corollary~\ref{C:mosco-weak}, in particular \eqref{eq:prox-conv}, we
  have that
  \[
    \lim_{\cK\ni k\to\infty} \|G_k(\bar{x},t)\|_X = \tfrac{1}{t}\|\bar{x}-\prox_{t\psi_{\bar x}}(\bar{x}-t\nabla f_0(\bar{x}))\|_X.
  \]
  Since $\bar{x}$ is stationary, Theorem~\ref{T:statpt} ensures
  that $\bar{x}=\prox_{t\psi_{\bar x}}(\bar{x}-t\nabla f_0(\bar{x}))$ and
  therefore
  \[
    \lim_{\cK\ni k\to\infty} \|x_k-\bar{x}\|_X = 0,
  \]
  proving the desired result.
\end{proof}

Theorem~\ref{T:strong-convex} demonstrates strong convergence when the weak
limit of $\{x_k\}_{\cK}$ is a stationary point.  When $f$ includes a Tikhonov
regularization term, we can prove that this assumption is satisfied.

\begin{theorem}\label{T:dual-weak}
  Let the assumptions of Corollary~\ref{C:conv1} hold and assume that
  \[
    f_0(x) = \frac{\tau}{2}\|x\|_X^2 + \hat{f}_0(x),
  \]
  where $\tau>0$, $\hat{f}_0:X\to\real$ is Fr\'{e}chet differentiable, and
  $\nabla\hat{f}_0$, $f_1$ and $f_1'$ are completely continuous. Moreover,
  suppose there exists $\bar{x}\in X$ and a subsequence index set
  $\cK\subseteq\mathbb{N}$ such that
  $h_k\xrightarrow{\cK}0$ and $x_k\xrightharpoonup{\cK}\bar{x}$.  Then,
  $\bar{x}$ is a stationary point of \eqref{eq:optprob}.  In addition, if
  $\hat{f}_0$ is $\tau_0$-weakly convex with $\tau_0\in(0,\tau)$, i.e.,
  \[
    x\mapsto\frac{\tau_0}{2}\|x\|_X^2 + \hat{f}_0(x)
  \]
  is convex, then $x_k\xrightarrow{\cK}\bar{x}$.
\end{theorem}
\begin{proof}
  Set $t=\frac{1}{\tau}$. Using Condition~\ref{as:inexactGrad} and
  the assumed form of $f_0$, we have that
  \[
   \begin{aligned}
   \|(x_k-tg_k)-&(\bar{x}-t\nabla f_0(\bar{x}))\|_X \\
   &\le t\|g_k-\nabla f_0(x_k)\|_X + \|(x_k-t\nabla f_0(x_k))-(\bar{x}-t\nabla f_0(\bar{x})\|_X \\
   &= t\|g_k-\nabla f_0(x_k)\|_X + t\|\nabla\hat{f}_0(x_k) - \nabla\hat{f}_0(\bar{x})\|_X \\
   &\le t\kappa_{\rm grad}h_k + t\|\nabla\hat{f}_0(x_k) - \nabla\hat{f}_0(\bar{x})\|_X
   \end{aligned}
  \]
  and hence
  \[
    x_k-tg_k\xrightarrow{\cK}\bar{x}-t\nabla f_0(\bar{x}).
  \]
  Therefore, Corollary~\ref{C:mosco-weak} and the nonexpansivity of the
  proximity operator yield
  \[
    \prox_{t\psi_k}(x_k-tg_k)\xrightarrow{\cK}\prox_{t\psi_{\bar x}}(\bar{x}-t\nabla f_0(\bar{x})).
  \]
  Now, the assumption that $h_k\xrightarrow{\cK}0$ and the monotonicity
  properties of the proximal gradient \cite[Lemma~2]{baraldi2022proximal}
  ensure that
  \[
    \|\prox_{t\psi_k}(x_k-tg_k)-x_k\|_X\xrightarrow{\cK} 0.
  \]
  Therefore, the weak lower semicontinuity of the norm yields
  \[
    0 = \liminf_{\cK\ni k\to\infty} \|\prox_{t\psi_k}(x_k-tg_k)-x_k\|_X \ge \|\prox_{t\psi_{\bar x}}(\bar{x}-t\nabla f_0(\bar{x}))-\bar{x}\|_X
  \]
  and so $\bar{x}=\prox_{t\psi_{\bar x}}(\bar{x}-t\nabla f_0(\bar{x}))$,
  implying stationarity of $\bar x$. To conclude, if $\hat{f}_0$ is
  $\tau_0$-weakly convex with $\tau_0<\tau$, then $f_0$ is strongly convex and
  strong convergence follows from Theorem~\ref{T:strong-convex}.
\end{proof}

Our final convergence result builds on Theorem~\ref{T:dual-weak}, proving that
when the Hessian sequence $\{B_k\}$ is bounded and $J$ is convex, the entire
sequence of iterates $\{x_k\}$ converges to a stationary point of
\eqref{eq:optprob}.

\begin{corollary}\label{C:conv-lim-strong}
  Let the assumptions of Theorems~\ref{T:conv-lim}~and~\ref{T:dual-weak} hold.
  Moreover, suppose $J$ is convex and admits a unique minimizer $\bar{x}\in X$.
  If $\{x_k\}$ is bounded, then $x_k\to\bar{x}$.
\end{corollary}
\begin{proof}
  Let $\tilde{x}$ be any weak accumulation point of $\{x_k\}$, which exists since
  $\{x_k\}$ is bounded, and let $\cK\subseteq\mathbb{N}$ be a subsequence index
  set on which $x_k\xrightharpoonup{\cK}\tilde{x}$.  By Theorem~\ref{T:conv-lim},
  $h_k\xrightarrow{\cK}0$ and by Theorem~\ref{T:dual-weak}, we have that
  $x_k\xrightarrow{\cK}\tilde{x}$ and $\tilde{x}$ is a stationary point	of
  \eqref{eq:optprob}.  Convexity of $J$ ensures that the stationarity
  conditions \eqref{eq:statpt} are necessary and sufficient for optimality, so
  $\tilde{x}=\bar{x}$.  Furthermore, boundedness of $\{x_k\}$ ensures that
  every subsequence has a further weakly converging subsequence with limit
  $\bar{x}$.  Consequently, $x_k\to\bar{x}$.
\end{proof}

\section{Numerics}\label{sec:num}
In this section, we demonstrate the numerical performance of
Algorithm~\ref{alg:tr} on various examples from PDE-constrained optimization.
Throughout, $D\subset\real^d$, $d=1,2$, is the physical domain with boundary
$\partial D$ and $(\Omega,\mathcal{F},\mathbb{P})$ is a probability space with
set of outcomes $\Omega$, $\sigma$-algebra of events
$\mathcal{F}\subseteq 2^\Omega$ and probability measure
$\mathbb{P}:\mathcal{F}\to[0,1]$.  The optimization space in all examples is
$X=L^2(D)$, while the space $Y$ in the first example is
$Y=L^2(\Omega,\mathcal{F},\mathbb{P})$ and $Y=\real$ in the second.
In both examples, $f_1$ is the composition of a function $\hat{f}_1$ with
the PDE solution map $S(x)$ at $x\in X$.  In order to evaluate the Jacobian
of $f_1=\hat{f}_1\circ S$, we employ the adjoint approach
\cite[Section~1.6.2]{MHinze_RPinnau_MUlbrich_SUlbrich_2009}.  Furthermore, we
choose $B_k=\nabla^2 L(x_k,\theta_k)$ and apply the Hessian of $f_1$ as in
\cite[Section~1.6.5]{MHinze_RPinnau_MUlbrich_SUlbrich_2009}.  Recall that each
gradient computation requires a linear adjoint PDE solve, while each Hessian
application requires two additional linearized PDE solves.

For our numerical results, we employ the truncated CG subproblem
solver developed in \cite[Algorithm~4]{baraldi2024efficient} with the number of
iterations capped at 15, lower and upper safeguards $10^{-12}$ and $10^{12}$,
respectively, descent parameter $10^{-4}$, and relative and absolute tolerances
$10^{-2}$ and $10^{-4}$, respectively. In the notation of
\cite[Algorithm~4]{baraldi2024efficient}, these parameters are {\tt maxit},
$t_{\min}$, $t_{\max}$, $\eta$, $\tau_k$ and $\bar\tau$, respectively.
Note that this choice of subproblem solver ensures that the Cauchy point
step lengths $t_k$ satisfy $10^{-12}=t_{\min} \le t_k \le t_{\max}=10^{12}$ for
all $k$. In addition, we use Algorithm~\ref{alg:prox-phik} to compute the
subproblem proximity operator. For this, we set $\lambda_0=1$,
$\gamma_{\min}=10^{-6}$, $\gamma_{\max}=10^6$, $a_{\min}=1$ (i.e., monotonic
line search), $\sigma_1=0.1$, $\sigma_2=0.9$, $\alpha=10^{-4}$,
and $\beta=0.5$.  We terminate Algorithm~\ref{alg:prox-phik} when
\begin{equation}\label{eq:alg2-tol}
  \frac{\|s^{(n)}\|_Y}{\gamma^{(n)}} \le 10^{-10}.
\end{equation}
For Algorithm~\ref{alg:tr}, we set $\Delta_1=10$, $\eta_1=10^{-4}$,
$\eta_2=0.5$, $\gamma_1=\gamma_2=0.25$, and $\gamma_3=10$.  Finally, we
employ the initial guess $x_1=0$ and terminate Algorithm~\ref{alg:tr} when
\begin{equation}\label{eq:stopcond}
  h_k \le 10^{-8}.
\end{equation}
A consequence of Theorem~\ref{T:conv} is that \eqref{eq:stopcond} is guaranteed
to be satisfied after finitely many iterations.

\subsection{Risk-Averse Control of Burgers' Equation}
\label{ss:burgers}

The goal of this example is to compare our approach with the
primal-dual risk minimization algorithm \cite{kouri2022primal}
and to demonstrate its ability to control inexact PDE solves via
Conditions~\ref{as:inexactGrad} and \ref{as:inexactObj}.
Let $D=(0,1)$ and $\tau>0$, and consider the optimization problem
\begin{equation}
  \min_{z\in L^2(D)}\; \risk\left(\tfrac{1}{2}\|S(z)-1\|_{L^2(D)}^2\right) + \tfrac{\tau}{2}\|z\|_{L^2(D)}^2,
\end{equation}
where $u=S(z):\Omega\to H^1(D)$ solves the weak form of Burgers' equation
\begin{subequations}\label{eq:burgers}
\begin{align}
  -\nu(\omega)\partial_{xx} u(\omega) + u(\omega)\partial_x u(\omega) &= f(\omega)+z &&\text{in $D$ a.s.} \\
  [u(\omega)](0) = d_0(\omega), \quad [u(\omega)](1) &= d_1(\omega) &&\text{a.s.}
\end{align}
\end{subequations}
Here, $\nu$, $f$, $d_0$ and $d_1$ are the random functions specified in
\cite{kouri2022primal}. For this example, $f_0(z)=\frac{\tau}{2}\|z\|^2_{L^2(D)}$,
$f_1(z)=\frac{1}{2}\|S(z)-1\|^2_{L^2(D)}$, $\phi_0\equiv 0$ and $\phi_1=\risk$
is a convex combination of the expected value and the average
value-at-risk
\[
  \risk(y) = (1-\lambda)\bbe[y] + \lambda\CVaR_p[y]
\]
with $\lambda=0.75$ and $p=0.9$, i.e., $\risk=\sigma_{\riskenv}$ with
$\riskenv=\{\theta\in L^2(\Omega,\mathcal{F},\mathbb{P})\,\vert\, \mathbb{E}[\theta]=1,\;\;0.25 \le\theta\le 7.75\;\text{a.s.}\}$.
We discretize the state $u$ and the control
$z$ in space using continuous piecewise linear finite elements on a uniform
mesh with 257 intervals and solve the resulting nonlinear system of equations
using Newton's method globalized with a backtracking line search.
Finally, we approximate the risk measure $\risk$ using
sample average approximation (SAA) with 10,000 Monte Carlo samples.
We note that Newton's method for solving \eqref{eq:burgers} for a fixed Monte Carlo sample, which we warmstart using
the solution at the previous sample, required 3.1 iterations on average in our
numerical experiments.  Consequently, solving \eqref{eq:burgers} is roughly
three times more expensive than solving the adjoint PDE when
evaluating the gradient.

We summarize the performance of Algorithm~\ref{alg:tr} in
Table~\ref{tbl:burgers} and include the performance of the primal-dual risk
minimization algorithm ({\tt PD-Risk}) \cite{kouri2022primal} and the bundle
method ({\tt Bundle}) described in \cite{HSchramm_JZowe_1992a} for comparison.
The results for {\tt PD-Risk} and {\tt Bundle} can also be found in
\cite[Tables~6~\&~7]{kouri2022primal}.
\begin{table}[!ht]
\centering
\begin{tabular}{l | r r r r}
  {\tt method} & {\tt iter} & {\tt nfval} & {\tt ngrad} & {\tt nhess} \\
  \hline
  {\tt Alg.~\ref{alg:tr}}  & {\tt   7} & {\tt   8} & {\tt   8} & {\tt 119} \\
  \hline
  {\tt PD-Risk} & {\tt  8} & {\tt  46} & {\tt  44} & {\tt 128} \\
  {\tt Bundle}  & {\tt 69} & {\tt 182} & {\tt 182} & {\tt ---}
\end{tabular}
\caption{Comparison of Algorithm~\ref{alg:tr} with the primal-dual risk
  minimization algorithm {\tt PD-Risk} and a bundle method {\tt Bundle}:
  {\tt iter} is the number of iterations, {\tt nfval} is the number of function
  evaluations, {\tt ngrad} is the number of gradient computations and
  {\tt nhess} is the number of Hessian applications.}
\label{tbl:burgers}
\end{table}
In Table~\ref{tbl:burgers}, {\tt iter} is the number of iterations, {\tt nfval}
is the number of evaluations of $f_0$ and $f_1$, {\tt ngrad} is the number of
derivative evaluations for $f_0$ and $f_1$, and {\tt nhess} is the number of
Hessian applications for $f_0$ and $f_1$.  Note that each derivative evaluation
of $f_1$ requires 10,000 deterministic adjoint PDE solves and each Hessian
application requires 20,000 additional deterministic linearized PDE solves.
Consequently, Algorithm~\ref{alg:tr} required 80,000 deterministic nonlinear
PDE solves to evaluate the objective function value and 2,460,000 deterministic
linear PDE solves to evaluate and apply the derivatives. In contrast,
{\tt PD-Risk} required 460,000 deterministic nonlinear PDE solves and 3,000,000
deterministic linear PDE solves, while {\tt Bundle} required 1,820,000
deterministic nonlinear and linear PDE solves. Although {\tt Bundle} performed
fewer linear PDE solves to compute derivatives, each nonlinear PDE solve
requires multiple linear solves.  When accounting for the linearized solves
required to solve the state, {\tt PD-Risk} required a total of 4,824,834
deterministic linear PDE solves, compared to 2,708,194 for
Algorithm~\ref{alg:tr}, and so Algorithm~\ref{alg:tr} reduced the total number
of deterministic linear PDE solves by a factor of $1.78$.
Algorithm~\ref{alg:tr} also provides additional benefits over {\tt PD-Risk}
since it can leverage inexact function and derivative evaluations, which we
demonstrate next.

Let $c_n(u,z)=0$ denote the discretized weak form of \eqref{eq:burgers} and
$u^{n,\ell}_k$ denote the $\ell$-th iterate generated by Newton's method for
solving $c_n$ at the $n$-th sample and the $k$-th iteration of
Algorithm~\ref{alg:tr}.
Table~\ref{tbl:burgers-hist} includes the iteration histories for two instances
of Algorithm~\ref{alg:tr}: The top half corresponds to the iteration history
using accurate state PDE solves that we terminate based on the criterion
\[
  \|c_n(u_k^{n,\ell},z_k)\| \le \tau\max\{1,\|c_n(u_k^{n,0},z_k)\|\},
\]
where $\tau=10^{-4}\sqrt{\epsilon_{\rm mach}}\approx 1.49\times 10^{-12}$
and $\epsilon_{\rm mach}$ denotes machine precision, while the lower half
is the iteration history where we have allowed Algorithm~\ref{alg:tr} to
supply $\tau$ using Conditions~\ref{as:inexactGrad} and \ref{as:inexactObj}
with $\kappa_{\rm grad}=\kappa_{\rm val}=\kappa_{\rm jac}=1$ and
$\kappa_{\rm obj}=10^4$. We selected this value for $\kappa_{\rm obj}$ so that
the initial tolerances generated by Conditions~\ref{as:inexactGrad} and
\ref{as:inexactObj} were the same order of magnitude.
In Table~\ref{tbl:burgers-hist}, the column labelled {\tt val tol} contains the
tolerances from Condition~\ref{as:inexactObj}, {\tt grad tol} contains the
subproblem model tolerances from Condition~\ref{as:inexactGrad} and
{\tt itsp} contains the number of subproblem solver iterations.
Both the exact and inexact algorithms perform comparably, requiring
the same number of iterations. However, the inexact variant leverages
significantly relaxed PDE solver tolerances, reducing the total number of
Newton iterations from 248,194 to 216,706.
Finally, to evaluate the proximity operator of $\psi_k$, the average iteration
counts of Algorithm~\ref{alg:prox-phik} were 26.8 and 28.3 for the
high-fidelity and adaptive-tolerance experiments, respectively.  Consequently, only modest computational effort is required to achieve the prescribed tolerance in \eqref{eq:alg2-tol} for this example.
\begin{table}[!ht]
\small
\begin{tabular}{r r r r r r r r r}
    $k$ &        $J(x_k)$ &           $h_k$ &      $\Delta_k$ & $\|x_k^+-x_k\|$ &  {\tt val tol} & {\tt grad tol} & {\tt itsp} \\
\hline
{\tt  0} & {\tt 2.4289e-2} & {\tt 3.9059e-3} & {\tt  1.0000e+1} & {\tt       ---} & {\tt 1.490e-12} & {\tt 1.490e-12} & {\tt  ---} \\
{\tt  1} & {\tt 1.3441e-2} & {\tt 3.5555e-3} & {\tt  1.0000e+2} & {\tt 1.1950e+0} & {\tt 1.490e-12} & {\tt 1.490e-12} & {\tt   15} \\
{\tt  2} & {\tt 1.2453e-2} & {\tt 2.5784e-3} & {\tt  1.0000e+3} & {\tt 3.5392e-1} & {\tt 1.490e-12} & {\tt 1.490e-12} & {\tt   15} \\
{\tt  3} & {\tt 1.2151e-2} & {\tt 6.9818e-4} & {\tt  1.0000e+4} & {\tt 2.6859e-1} & {\tt 1.490e-12} & {\tt 1.490e-12} & {\tt   15} \\
{\tt  4} & {\tt 1.2129e-2} & {\tt 6.1039e-5} & {\tt  1.0000e+5} & {\tt 1.3944e-1} & {\tt 1.490e-12} & {\tt 1.490e-12} & {\tt   15} \\
{\tt  5} & {\tt 1.2128e-2} & {\tt 1.6595e-6} & {\tt  1.0000e+6} & {\tt 1.8164e-2} & {\tt 1.490e-12} & {\tt 1.490e-12} & {\tt   15} \\
{\tt  6} & {\tt 1.2128e-2} & {\tt 3.9913e-7} & {\tt  1.0000e+7} & {\tt 5.4855e-4} & {\tt 1.490e-12} & {\tt 1.490e-12} & {\tt   15} \\
{\tt  7} & {\tt 1.2128e-2} & {\tt 3.2924e-9} & {\tt  1.0000e+8} & {\tt 4.5861e-5} & {\tt 1.490e-12} & {\tt 1.490e-12} & {\tt   15} \\
\hline
{\tt  0} & {\tt 2.3866e-2} & {\tt 4.4689e-3} & {\tt  1.0000e+1} & {\tt       ---} & {\tt  1.000e-2} & {\tt 4.469e-3} & {\tt  ---} \\
{\tt  1} & {\tt 1.3256e-2} & {\tt 2.2653e-3} & {\tt  1.0000e+2} & {\tt 1.2599e+0} & {\tt  1.931e-3} & {\tt 2.265e-3} & {\tt   15} \\
{\tt  2} & {\tt 1.2478e-2} & {\tt 5.5037e-4} & {\tt  1.0000e+3} & {\tt 3.1316e-1} & {\tt  1.191e-4} & {\tt 5.504e-3} & {\tt   15} \\
{\tt  3} & {\tt 1.2132e-2} & {\tt 3.9939e-4} & {\tt  1.0000e+4} & {\tt 3.0440e-1} & {\tt  5.023e-5} & {\tt 3.994e-4} & {\tt   15} \\
{\tt  4} & {\tt 1.2128e-2} & {\tt 1.1822e-5} & {\tt  1.0000e+5} & {\tt 5.6905e-2} & {\tt  3.581e-7} & {\tt 1.182e-5} & {\tt   15} \\
{\tt  5} & {\tt 1.2128e-2} & {\tt 7.3445e-6} & {\tt  1.0000e+6} & {\tt 4.3233e-3} & {\tt 7.582e-10} & {\tt 7.344e-6} & {\tt   15} \\
{\tt  6} & {\tt 1.2128e-2} & {\tt 2.6303e-7} & {\tt  1.0000e+7} & {\tt 7.6969e-4} & {\tt 6.844e-11} & {\tt 2.630e-7} & {\tt   15} \\
{\tt  7} & {\tt 1.2128e-2} & {\tt 3.2007e-9} & {\tt  1.0000e+8} & {\tt 5.8116e-5} & {\tt 1.490e-12} & {\tt 3.201e-9} & {\tt   15} \\
\hline
\end{tabular}
\caption{Algorithm~\ref{alg:tr} iteration history using high-fidelity (rows 2
         through 8) and inexact (rows 9 through 15) PDE solves.  Columns 6 and
         7 list the relative residual tolerance determined by the trust-region
         algorithm.}
\label{tbl:burgers-hist}
\end{table}

\subsection{Sparse Optimal Control}
The goal of this example is to demonstrate the mesh independence of
Algorithm~\ref{alg:tr}.  Let $D=(0,0.6)\times(0,0.2)$.  We consider a
deterministic semilinear elliptic optimal control problem motivated by the one
studied in \cite{kouri2020risk}. Our control problem is given by
\begin{equation}\label{eq:sparse-ctrl}
  \min_{\substack{z\in L^2(D) \\ -10\le z\le 10}}
  \max\left\{0,w-\frac{1}{|D_o|}\int_{D_o} S(z)\,\textup{d}x\right\} + \frac{\tau}{2}\|z\|_{L^2(D)}^2 + \tau_1 \|z\|_{L^1(D)},
\end{equation}
where $S(z)=u\in H^1(D)$ solves the weak form of the semilinear elliptic PDE
\begin{align}
  -\kappa\Delta u + \gamma u^3 &= 12\chi_{D_b}+z && \text{in $D$} \nonumber \\
  u &=0 &&\text{on $\Gamma = [0,0.6]\times\{0\}$} \label{eq:sparse-ctrl-pde}\\
  \kappa\nabla u\cdot n &= 0 && \text{on $\partial D\setminus\Gamma$}. \nonumber
\end{align}
Here, $D_b=(0,0.1)\times(0.167,0.2)$, $D_o=(0.5,0.6)\times(0.167,0.2)$,
$\chi_{D_b}$ denotes the characteristic function of $D_b$, $\tau=10^{-4}$,
$\tau_1=10^{-2}$, $\kappa=0.25$, $\gamma=1.45$ and $w=0.2$. For this problem,
$f_0(z)=\frac{\tau}{2}\|z\|^2_{L^2(D)}$, $\phi_1(t)=\max\{0,t\}$ and
\[
  f_1(z) = w-\frac{1}{|D_o|}\int_{D_o} S(z)\,\textup{d}x.
\]
Moreover, $\phi_0$ is the sum of the $L^1$-penalty term and the indicator function
associated with the bound constraints $-10\le z\le 10$ a.e.  We discretized $u$
in \eqref{eq:sparse-ctrl-pde} using P1 finite elements on a uniform triangular
mesh and $z$ using piecewise constants on the same mesh.  As in the previous
example, we solve the discretized nonlinear PDE using a line search globalized
Newton's method.  To investigate mesh independence, we performed a mesh
refinement study, where we uniformly refined the initial mesh and solved the
refined optimization problem from an initial guess of $z\equiv0$. The initial
mesh was generated by splitting the elements of a 60-by-20 uniform
quadrilateral mesh into triangles. Subsequent meshes were generated by uniformly
refining the initial quadrilateral mesh and then partitioning into triangles.
\begin{table}[!ht]
\centering
\small
\begin{tabular}{c | r r r r r r r}
  {\tt mesh} & {\tt iter} & {\tt nfval} & {\tt ngrad} & {\tt nhess} & {\tt npsi} & {\tt nprox} & {\tt aprox} \\
\hline
{\tt    60$\times$20} & {\tt 2} & {\tt 3} & {\tt 3} & {\tt 34} & {\tt 1875} & {\tt 798} & {\tt 20.80} \\
{\tt\!\!\!120$\times$40} & {\tt 2} & {\tt 3} & {\tt 3} & {\tt 34} & {\tt 2084} & {\tt 772} & {\tt 20.06} \\
{\tt\!\!\!240$\times$80} & {\tt 2} & {\tt 3} & {\tt 3} & {\tt 34} & {\tt 2176} & {\tt 795} & {\tt 20.71} \\
{\tt  480$\times$160} & {\tt 2} & {\tt 3} & {\tt 3} & {\tt 34} & {\tt 2192} & {\tt 788} & {\tt 20.51} \\
{\tt  960$\times$320} & {\tt 2} & {\tt 3} & {\tt 3} & {\tt 34} & {\tt 2079} & {\tt 789} & {\tt 20.54} \\
{\tt\!\!\!1920$\times$640} & {\tt 2} & {\tt 3} & {\tt 3} & {\tt 34} & {\tt 2138} & {\tt 787} & {\tt 20.49} \\
\hline
\end{tabular}
\caption{Mesh refinement study of Algorithm~\ref{alg:tr} applied to the sparse
control application, demonstrating mesh independence.}
\label{tbl:sparse-ctrl-mesh}
\end{table}
Table~\ref{tbl:sparse-ctrl-mesh} lists the results of this study.  The columns
of Table~\ref{tbl:sparse-ctrl-mesh} include the total number of iterations
({\tt iter}), the number of evaluations of $f_0$ and $f_1$
({\tt nfval}), the number of derivative computations for $f_0$ and $f_1$
({\tt ngrad}), the number of Hessian applications for $f_0$ and $f_1$
({\tt nhess}), the number of evaluations of $\psi_k$ ({\tt npsi}), the number
of proximity operator evaluations for $\psi_k$ ({\tt nprox}) and the average
number of iterations of Algorithm~\ref{alg:prox-phik} ({\tt aprox}).
As seen in Table~\ref{tbl:sparse-ctrl-mesh}, the performance of
Algorithm~\ref{alg:tr} is essentially constant, suggesting that
Algorithm~\ref{alg:tr} is mesh independent.  To achieve this, our
implementation uses infinite-dimensionally consistent scaled inner products and
norms for $X$ and $Y$.  Note that quantities like the gradient and proximity
operator depend on the choice of inner product, which can complicate their
computation in general.  However, in this example, the discretized inner product for $X$ is
the Euclidean inner product scaled by the mesh cell volume. In
addition, we see that the average number of iterations of
Algorithm~\ref{alg:prox-phik} to compute the proximity operator of $\psi_k$ is
relatively modest, requiring roughly 20 iterations on average to achieve the
requested tolerance of $10^{-10}$ in \eqref{eq:alg2-tol}.

\section{Conclusions}
We have introduced a provably convergent inexact trust-region algorithm,
extended from \cite{baraldi2022proximal}, for the structured class of nonsmooth
optimization problems represented by \eqref{eq:optprob}.  This class
encapsulates many important applications including risk-averse stochastic
optimization and nonsmooth penalty methods for nonlinear programming.
We verify through numerical examples that our method is quite efficient
when solving discretized PDE-constrained optimization problems, demonstrating
invariance to the discretization size and outperforming some existing nonsmooth
methods.  Although our method performs well in practice, the theoretical
results rely heavily on the Lipschitz continuity of $\phi_1$ and the exact
computation of the proximity operator $\prox_{t\psi_k}$.  Permitting
extended real valued $\phi_1$ would enable, e.g., the solution of
nonlinearly constrained problems, while permitting inexact proximity operator
evaluation may further improve efficiency.  In future work, we hope
to relax these requirements.

\backmatter

\bmhead{Acknowledgements}

We would like to thank Thomas Surowiec for the many fruitful discussions about
earlier versions of this work.
Sandia National Laboratories is a multimission laboratory
managed and operated by National Technology and Engineering
Solutions of Sandia, LLC., a wholly owned subsidiary of
Honeywell International, Inc., for the U.S.\ Department of
Energy’s National Nuclear Security Administration under
contract DE-NA0003525.
This paper describes objective technical results and analysis. Any
subjective views or opinions that might be expressed in the paper
do not necessarily represent the views of the U.S.\ Department of
Energy or the United States Government.

\section*{Declarations}

\begin{itemize}
\item {\bf Funding:} This research was sponsored, in part, by the
  U.S.~Department of Energy Office of Science under the Early Career Research
  Program and the U.S.~Air~Force Office of Scientific Research.
\item {\bf Competing Interests:} The author has no competing interests to
  declare that are relevant to the content of this article.
\item {\bf Data Availability:} No data was generated or used for this article. 
\end{itemize}

\begin{appendices}
\section{Weak Convergence of the Spectral Proximal Gradient Method}
\label{ap:spg}
In this appendix, we prove that the spectral proximal gradient method used 
in Algorithm~\ref{alg:prox-phik} produces weakly converging iterates.  Let
$Z$ be a real Hilbert space, $f:Z\to\real$ and $\phi:Z\to(-\infty,+\infty]$
be convex functions, and consider the minimization problem
\begin{equation}\label{eq:optprob-appendix}
  \min_{z\in Z} \{F(z)\coloneqq f(z)+\phi(z)\}.
\end{equation}
We require that the functions $f$ and $\phi$ satisfy the following assumptions.
\begin{assumption}\label{as:appendix}
The problem data in \eqref{eq:optprob-appendix} satisfies the following
conditions.
\begin{enumerate}
  \item The function $\phi:Z\to(-\infty,+\infty]$ is proper, closed and
        convex.
  \item The function $f:Z\to\real$ is convex and Fr\'{e}chet differentiable on
        an open set containing $\dom{\phi}$ with Lipschitz continuous gradient.
  \item The objective function $F=f+\phi$ is bounded from below.
\end{enumerate}
\end{assumption}
In relation to Algorithm~\ref{alg:prox-phik}, $Z=Y$, $f(z)=-d(z)$ and
$\phi(z)=\phi_1^*(z)$.  The main result of this appendix demonstrates that the
iterates generated by the spectral proximal gradient method listed in
Algorithm~\ref{alg:spg} converge weakly to a solution of
\eqref{eq:optprob-appendix}.
\begin{algorithm}[!ht]
\caption{Spectral Proximal Gradient Method}
  \label{alg:spg}
  \begin{algorithmic}[1]
    \Require{Initial guess $z^{(1)}\in\dom{\phi}$,
      initial value $\mathfrak{F}^{(n)}=F(z^{(1)})$,
      safeguards $0<\gamma_{\min}<\gamma_{\max}<+\infty$ and $a_{\min}\in(0,1]$,
      initial steplengths $\gamma^{(1)}\in[\gamma_{\min},\gamma_{\max}]$ and
      $\lambda_0\in(0,1]$, and $0<\sigma_1<\sigma_2<1$, and positive parameters
      $\alpha\in(0,1)$, and $\beta\in[\sigma_1,\sigma_2]$}
    \For{$n=1,2,\ldots$}
      \State{$\mu^{(n)}\gets\nabla f(z^{(n)})$}
      \If{$n>1$}
        \State{$\gamma^{(n)}\gets\min\left\{\gamma_{\max},\max\left\{\gamma_{\min},\displaystyle{\frac{\lambda^{(n-1)}\|s^{(n-1)}\|_Z^2}{(\mu^{(n)}-\mu^{(n-1)},s^{(n-1)})_Z}}\right\}\right\}$}
      \EndIf
      \State{$s^{(n)}\gets\prox_{\gamma^{(n)}\phi}(z^{(n)}-\gamma^{(n)}\mu^{(n)})-z^{(n)}$}
      \State{$\lambda^{(n)}\gets\lambda_0$}
      \State{$\mathcal{Q}^{(n)}\gets(\mu^{(n)},s^{(n)})_Z+\phi(z^{(n)}+s^{(n)})-\phi(z^{(n)})$}
      \While{$F(z^{(n)}+\lambda^{(n)} s^{(n)}) > \mathfrak{F}^{(n)} + \alpha\lambda^{(n)}\mathcal{Q}^{(n)}$}
        \State{$\delta\gets-\displaystyle{\frac{[\lambda^{(n)}(\mu^{(n)},s^{(n)})_Z+\phi(z^{(n)}+\lambda^{(n)}s^{(n)})-\phi(z^{(n)})]}{2[f(z^{(n)}+\lambda^{(n)} s^{(n)})-f(z^{(n)})-\lambda^{(n)}(\mu^{(n)},s^{(n)})_Z]}}$}
        \If{$\delta\in[\sigma_1,\sigma_2]$}
          \State{$\lambda^{(n)}\gets\delta\lambda^{(n)}$}
        \Else
          \State{$\lambda^{(n)}\gets\beta\lambda^{(n)}$}
        \EndIf
      \EndWhile
      \State{$z^{(n+1)}\gets z^{(n)} + \lambda^{(n)}s^{(n)}$}
      \State{$\mathfrak{F}^{(n+1)}\gets (1-a^{(n)})\mathfrak{F}^{(n)}+a^{(n)}F(z^{(n+1)})$ for $a^{(n)}\in[a_{\min},1]$}
    \EndFor
  \end{algorithmic}
\end{algorithm}
However, we first prove that the line search in Algorithm~\ref{alg:spg} (i.e.,
lines~9-16) is well defined.  Note that $\delta$ at line~10 is the minimizer of
a quadratic interpolant of the map $t\mapsto F(z^{(n)}+t\lambda^{(n)}s^{(n)})$,
where $\{z^{(n)}\}\subset Z$ is the sequence of iterates generated by
Algorithm~\ref{alg:spg}, $\{s^{(n)}\}\subset Z$ is the sequence of proximal
gradient steps and $\{\lambda^{(n)}\}\subset[0,1]$ is the sequence of step
lengths. In addition, Algorithm~\ref{alg:spg} leverages the nonmonotone line
search from \cite{zhang2004nonmonotone} at line~9 instead of the max-based line
search from \cite{grippo1986nonmonotone} because it yields the useful property
that $\lambda^{(n)}\|z^{(n)}-z^{(n-1)}\|_Z^2\to 0$.

\begin{proposition}\label{prop:spg-appendix}
  Let Assumption~\ref{as:appendix} hold and suppose that $z^{(n)}$ is not a
  solution to \eqref{eq:optprob-appendix}, then lines~9-16 of
  Algorithm~\ref{alg:spg} terminate after finitely many iterations with
  $\lambda^{(n)}>0$.
\end{proposition}
\begin{proof}
  We prove this result by induction.  Let $n=1$, since $z^{(1)}$ is not a
  solution to \eqref{eq:optprob-appendix}, $\|s^{(1)}\|_Z>0$ and so
  $\mathcal{Q}^{(1)}<0$. Moreover, we have that $\mathfrak{F}^{(1)}=F(z^{(1)})$
  and lines~9-16 of Algorithm~\ref{alg:spg} produce a sequence of iterates
  $\{\lambda_\ell^{(1)}\}$ with
  $\lambda_{\ell}^{(1)}\in[\sigma_1\lambda_{\ell-1}^{(1)},\sigma_2\lambda_{\ell-1}^{(1)}]$.
  For a contradiction, suppose that lines~9-16 do not terminate in finitely
  many iterations, then since $\sigma_2<1$, we have that
  $\lambda_\ell^{(1)}\to 0$ as $\ell\to\infty$.  Moreover, line~9 ensures that
  \begin{equation}\label{eq:bnd-appendix}
    \frac{F(z^{(1)}+\lambda_\ell^{(1)}s^{(1)}) - F(z^{(1)})}{\lambda_\ell^{(1)}}
    \ge \frac{F(z^{(1)}+\lambda_\ell^{(1)}s^{(1)}) - \mathfrak{F}^{(1)}}{\lambda_\ell^{(1)}}
    > \alpha\mathcal{Q}^{(1)}.
  \end{equation}
  The convexity of $\phi$ and differentiability of $f$ then yields
  \[
    \lim_{\ell\to\infty} \frac{F(z^{(1)}+\lambda_\ell^{(1)}s^{(1)})-F(z^{(1)})}{\lambda_\ell^{(1)}}
    \le (\nabla f(z^{(1)}),s^{(1)})_Z + \phi(z^{(1)}+s^{(1)}) - \phi(z^{(1)})
    = \mathcal{Q}^{(1)}.
  \]
  Hence, $\alpha\mathcal{Q}^{(1)}<\mathcal{Q}^{(1)}<0$, which leads to a
  contradiction since $\alpha\in(0,1)$. Whence, lines~9-16 terminate in
  finitely many iterations.  Now, let $n\in\mathbb{N}$ be such that $z^{(n)}$
  is not a solution to \eqref{eq:optprob-appendix} and suppose that
  $\lambda^{(n-1)}>0$, then
  \[
    \begin{aligned}
    \mathfrak{F}^{(n)}&=(1-a^{(n-1)})\mathfrak{F}^{(n-1)}+a^{(n-1)} F(z^{(n)}) \\
    &\ge (1-a^{(n-1)})(F(z^{(n)})-\alpha\lambda^{(n-1)}\mathcal{Q}^{(n-1)})+a^{(n-1)}F(z^{(n)}) \\
    &\ge F(z^{(n)}),
    \end{aligned}
  \]
  where the final inequality follows from \cite[Lemma~2]{baraldi2022proximal}.
  Consequently, \eqref{eq:bnd-appendix} holds with 1 replaced by $n$ and the
  analysis for the $n=1$ case applies, proving the desired result.
\end{proof}

In the subsequent theorem, we prove that $\{z^{(n)}\}$ weakly converges to
a solution of \eqref{eq:optprob-appendix}.  This result uses a notion of
convergence called quasi-Fej\'{e}r monotonicity. We refer the reader to
\cite[Chapter~5]{bauschke2017convex} for more details.

\begin{theorem}\label{th:spg-appendix}
  Let Assumption~\ref{as:appendix} hold and suppose that
  $\operatorname{arg\,min}F\neq \emptyset$.  Then $\{z^{(n)}\}$ converges
  weakly to a solution $\bar z$ of \eqref{eq:optprob-appendix}.
\end{theorem}
\begin{proof}
  Proposition~\ref{prop:spg-appendix} ensures that if there exists
  $n\in\mathbb{N}$ for which $z^{(n)}=\bar z$ is a solution to
  \eqref{eq:optprob-appendix}, then $z^{(n+j)}=\bar z$ for all
  $j\in\mathbb{N}$ and so $z^{(n)}$ converges to $\bar z$.  Now,
  suppose no such $n$ exists, then $\lambda^{(n)}>0$ for all $n$.
  In this case, the proof of weak convergence requires a generalization of
  \cite[Theorems~3.1~\&~3.2]{bello2016weak} to account for $\phi$ that is
  not necessarily an indicator function and to account for the nonmonotone
  line search.  Let $z\in Z$ be an arbitrary solution to
  \eqref{eq:optprob-appendix},
  $u^{(n)}\coloneqq z^{(n)}-\gamma^{(n)}\mu^{(n)}$ and
  $v^{(n)}\coloneqq \prox_{\gamma^{(n)}\phi}(u^{(n)})$.
  As in the proof of \cite[Lemma~3.1]{bello2016weak}, we have that
  \begin{align*}
  (z^{(n)}-&z^{(n+1)},z^{(n)}-z)_Z \\
     %=&\,-\lambda^{(n)}(s^{(n)}\pm\gamma^{(n)}\mu^{(n)},z^{(n)}-z\pm v^{(n)})_Z \\
     =&\,\lambda^{(n)}\gamma^{(n)}(\mu^{(n)},z^{(n)}-z)_Z
        +\lambda^{(n)}(v^{(n)}-u^{(n)},s^{(n)})_Z
        +\lambda^{(n)}(v^{(n)}-u^{(n)},z-v^{(n)})_Z \\
   \ge&\,\lambda^{(n)}\gamma^{(n)}(\mu^{(n)},z^{(n)}-z)_Z
        +\lambda^{(n)}(v^{(n)}-u^{(n)},s^{(n)})_Z
        +\lambda^{(n)}\gamma^{(n)}(\phi(v^{(n)})-\phi(z)) \\
   \ge&\,\lambda^{(n)}\gamma^{(n)}(f(z^{(n)})-f(z))
        +\lambda^{(n)}\gamma^{(n)}(\phi(v^{(n)})-\phi(z))
        +\lambda^{(n)}(v^{(n)}-u^{(n)},s^{(n)})_Z \\
     =&\,\lambda^{(n)}\gamma^{(n)}(F(z^{(n)})-F(z))
        +\lambda^{(n)}\gamma^{(n)}(\phi(v^{(n)})-\phi(z^{(n)}))
        +\lambda^{(n)}(v^{(n)}-u^{(n)},s^{(n)})_Z \\
   \ge&\,\lambda^{(n)}[\gamma^{(n)}(\phi(v^{(n)})-\phi(z^{(n)}))
      +\|^{(n)}\|_Z^2+\gamma^{(n)}(\mu^{(n)},s^{(n)})_Z] \\
     =&\,\lambda^{(n)}[\|s^{(n)}\|_Z^2 + \gamma^{(n)}\mathcal{Q}^{(n)}].
  \end{align*}
  Here, the first inequality follows from \cite[Lemma~1]{baraldi2022proximal},
  the second follows from the convexity of $f$, and the third from the
  optimality of $z$.  Consequently, we have that
  \begin{align*}
    \|z^{(n+1)}-z\|_Z^2
    &=\|z^{(n)}-z\|_Z^2+\|z^{(n+1)}-z^{(n)}\|_Z^2-2(z^{(n)}-z^{(n+1)},z^{(n)}-z)_Z \\
    %&\le \|z^{(n)}-z\|_Z^2 + \|z^{(n+1)}-z^{(n)}\|_Z^2 - 2\lambda^{(n)}[\|s^{(n)}\|_Z^2+\gamma^{(n)}\mathcal{Q}^{(n)}] \\
    &\le \|z^{(n)}-z\|_Z^2 + (\lambda^{(n)})^2\|s^{(n)}\|_Z^2 - 2\lambda^{(n)}[\|s^{(n)}\|_Z^2+\gamma^{(n)}\mathcal{Q}^{(n)}] \\
    &\le \|z^{(n)}-z\|_Z^2 - \lambda^{(n)}\|s^{(n)}\|_Z^2 - 2\lambda^{(n)}\gamma^{(n)}\mathcal{Q}^{(n)}. 
  \end{align*}
  Here, we have used the fact that
  $(\lambda^{(n)})^2-2\lambda^{(n)}\le-\lambda^{(n)}$ since
  $\lambda^{(n)}\in(0,1)$.  Moreover, \cite[Lemma~2]{baraldi2022proximal}
  ensures that
  \[
    -\mathcal{Q}^{(n)} \ge \frac{1}{\gamma^{(n)}}\|s^{(n)}\|_Z^2
    \quad\implies\quad
    \varepsilon^{(n)}\coloneqq-\lambda^{(n)}\|s^{(n)}\|_Z^2 - 2\lambda^{(n)}\gamma^{(n)}\mathcal{Q}^{(n)}
      \ge \lambda^{(n)}\|s^{(n)}\|_Z^2\ge 0.
  \]
  The construction of $\mathfrak{F}^{(n)}$ yields the bound
  \begin{align}
    \mathfrak{F}^{(n+1)} &= (1-a^{(n)})\mathfrak{F}^{(n)}+a^{(n)} F(z^{(n+1)}) \nonumber \\
    &\le (1-a^{(n)})\mathfrak{F}^{(n)}+a^{(n)}(\mathfrak{F}^{(n)}+\alpha\lambda^{(n)}\mathcal{Q}^{(n)}) \nonumber \\
    &= \mathfrak{F}^{(n)} + a^{(n)}\alpha\lambda^{(n)}\mathcal{Q}^{(n)} \label{eq:mon-appendix}
  \end{align}
  for all $n\ge 1$ and so
  \begin{equation}\label{eq:sum-appendix}
    -a_{\min}\alpha\sum_{n=1}^\infty\lambda^{(n)}\mathcal{Q}^{(n)}
    \le\sum_{n=1}^\infty \mathfrak{F}^{(n)}-\mathfrak{F}^{(n+1)}
    \le F(z^{(1)})-F(z) <\infty.
  \end{equation}
  Therefore, we have that
  \[
    0 \le \sum_{n=1}^\infty \varepsilon^{(n)} \le -2\frac{\gamma_{\max}}{a_{\min}\alpha} \sum_{n=1}^\infty \lambda^{(n)}\mathcal{Q}^{(n)} < \infty
  \]
  and $\{z^{(n)}\}$ is quasi-Fej\'{e}r monotone with respect to the set of
  minimizers of \eqref{eq:optprob-appendix}.  Consequently, $\{z^{(n)}\}$ is
  bounded and if every weak sequential cluster point is a minimizer to
  \eqref{eq:optprob-appendix}, then the iterates converge weakly to a
  minimizer \cite[Theorem~5.33]{bauschke2017convex}.

  To this end, \eqref{eq:mon-appendix} and \cite[Lemma~2]{baraldi2022proximal}
  (i.e., $\mathcal{Q}^{(n)}\le-\|s^{(n)}\|_Z^2/\gamma^{(n)}\le 0$) imply that
  $\{\mathfrak{F}^{(n)}\}$ is a decreasing sequence that is bounded below and
  hence convergent. This, \eqref{eq:sum-appendix}, and the fact that
  $\gamma^{(n)}\in[\gamma^{\min},\gamma^{\max}]$ yields
  \[
    \lim_{n\to\infty} \lambda^{(n)}\|s^{(n)}\|_Z^2 = 0.
  \]
  Let $\bar z\in Z$ be a weak accumulation point of $\{z^{(n)}\}$ and
  $\mathcal{N}\subseteq\mathbb{N}$ be a subsequence on which
  $z^{(n)}\rightharpoonup\bar z$. We split the result into two cases:
  $\{\lambda^{(n)}\}_{\mathcal{N}}$ does and does not converge to zero.  We
  first consider the case that $\{\lambda^{(n)}\}_{\mathcal{N}}$ does not
  converge to zero and so there exists a subsequence on which
  $\inf_{n\in\mathcal{N}}\lambda^{(n)}>0$. We redefine $\mathcal{N}$ to
  be this subsequence and therefore, on $\mathcal{N}$ we have that
  $s^{(n)}=v^{(n)}-z^{(n)}\to 0$ and the Lipschitz continuity of $\nabla f$
  ensures that $\nabla f(v^{(n)})-\nabla f(z^{(n)})\to 0$.  Now, consider the
  second case. According to lines~9-15 of Algorithm~\ref{alg:spg}, there exists
  $\lambda_-^{(n)}\in[\tfrac{1}{\sigma_2}\lambda^{(n)},\tfrac{1}{\sigma_1}\lambda^{(n)}]$
  such that
  \[
    F(z^{(n)})-F(z^{(n)}+\lambda_-^{(n)}s^{(n)})
    \le \mathfrak{F}^{(n)}-F(z^{(n)}+\lambda_-^{(n)}s^{(n)})
    < -\alpha\lambda_-^{(n)}\mathcal{Q}^{(n)},
  \]
  where the first inequality follows from line~9 and
  the fact that $\mathcal{Q}^{(n)}<0$, i.e.,
  \[
    \mathfrak{F}^{(n)} = (1-a^{(n-1)})\mathfrak{F}^{(n-1)}+a^{(n-1)}F(z^{(n)}) %\ge (1-p)(F(z^{(n)})-\alpha\lambda^{(n)}\mathcal{Q}^{(n)})+pF(z^{(n)})
    \ge F(z^{(n)}) - a^{(n-1)}\alpha\lambda^{(n)}\mathcal{Q}^{(n)}
    \ge F(z^{(n)}).
  \]
  Using convexity of $f$ and $\phi$, we have that
  \[
  \begin{aligned}
    -\alpha\lambda_-^{(n)}\mathcal{Q}^{(n)}
      >&\, \phi(z^{(n)})-\phi(z^{(n)}+\lambda_-^{(n)}s^{(n)})-\lambda_-^{(n)}(\nabla f(z^{(n)}+\lambda_-^{(n)}s^{(n)}),s^{(n)})_Z \\
    \ge&\, -\lambda_-^{(n)}\mathcal{Q}^{(n)} + \lambda_-^{(n)}(\nabla f(z^{(n)})-\nabla f(z^{(n)}+\lambda_-^{(n)}s^{(n)}),s^{(n)})_Z.
  \end{aligned}
  \]
  Consequently, the Cauchy-Schwarz inequality and
  \cite[Lemma~2]{baraldi2022proximal} yield
  \[
    \|\nabla f(z^{(n)})-\nabla f(z^{(n)}+\lambda_-^{(n)}s^{(n)})\|_Z\|s^{(n)}\|_Z > -(1-\alpha)\mathcal{Q}^{(n)}
    \ge \frac{(1-\alpha)}{\gamma^{(n)}}\|s^{(n)}\|_Z^2
  \]
  and so
  \[
    \|\nabla f(z^{(n)})-\nabla f(z^{(n)}+\lambda_-^{(n)}s^{(n)})\|_Z > \frac{(1-\alpha)}{\gamma^{\max}}\|s^{(n)}\|_Z.
  \]
  Notice that
  $z^{(n)}-(z^{(n)}+\lambda_-^{(n)}s^{(n)})=-\lambda_-^{(n)}s^{(n)}$ and thus
  \[
    \|z^{(n)}-(z^{(n)}+\lambda_-^{(n)}s^{(n)})\|_Z^2 = \|\lambda_-^{(n)}s^{(n)}\|_Z^2
      \le \frac{1}{\sigma_1} \lambda^{(n)} \|s^{(n)}\|_Z^2 \to 0.
  \]
  Here, we have used the fact that $0\le \lambda_-^{(n)}\le 1$, which ensures
  $(\lambda_-^{(n)})^2\le\lambda_-^{(n)}$.  Consequently,
  $\nabla f(z^{(n)})-\nabla f(z^{(n)}+\lambda_-^{(n)}s^{(n)})\to 0$
  and $\|s^{(n)}\|_Z\to 0$.

  To summarize, in both cases, there exists a subsequence
  $\mathcal{N}\subseteq\mathbb{N}$ on which $z^{(n)}\rightharpoonup\bar z$ and
  $s^{(n)}\to 0$.  Therefore, $z^{(n)}-v^{(n)}=s^{(n)}\to 0$ and
  $\nabla f(z^{(n)})-\nabla f(v^{(n)})\to 0$ on $\mathcal{N}$.  Now, by
  construction of $u^{(n)}$ and $v^{(n)}$, we have that
  \[
  \begin{aligned}
    \frac{1}{\gamma^{(n)}}&(u^{(n)}-v^{(n)}) \in \partial\phi(v^{(n)}) \\
    &\iff\quad
    -\frac{1}{\gamma^{(n)}}s^{(n)} + (\nabla f(z^{(n)})-\nabla f(v^{(n)})) \in \nabla f(v^{(n)}) + \partial\phi(v^{(n)}) = \partial F(v^{(n)}).
  \end{aligned}
  \]
  Since the left-hand side (strongly) converges to zero and
  $v^{(n)}=z^{(n)}+s^{(n)}\rightharpoonup\bar z$ on $\mathcal{N}$, it follows
  from \cite[Proposition~16.36]{bauschke2017convex} that
  $0\in\partial F(\bar z)$. Therefore, $\bar z$ is a solution to
  \eqref{eq:optprob-appendix} and weak convergence follows from quasi-Fej\'{e}r
  monotonicity \cite[Theorem~5.33]{bauschke2017convex}.
\end{proof}

\end{appendices}

\bibliography{references}

\end{document}